\documentclass{amsart}
\usepackage{amssymb,amsmath,latexsym,enumerate, dsfont,xspace}
\usepackage{graphicx,verbatim,enumerate,%
ifpdf,mdwlist}
\usepackage{color, soul}

\usepackage[normalem]{ulem}
\usepackage{mathabx}
\ifpdf
\usepackage{hyperref}
\else
\usepackage[hypertex]{hyperref}
\fi
\usepackage{graphicx}
\usepackage{cancel}
\usepackage{float}
\renewcommand{\phi}{\varphi}

\newtheorem{thm}{Theorem}[section]
\newtheorem{cor}[thm]{Corollary}
\newtheorem{lem}[thm]{Lemma}

\newtheorem{prop}[thm]{Proposition}
\newtheorem{obs}[thm]{Observation}
\newtheorem{claim}[thm]{Claim}

\theoremstyle{definition}

\newtheorem{defn}[thm]{Definition}

\newenvironment{notation}
{\begin{trivlist} \item[] {\bf Notation}. } {\end{trivlist}}

\DeclareMathOperator{\ran}{ran}

\renewcommand{\phi}{\varphi}

\newcommand{\abs}[1]{\lvert#1\rvert}
\newcommand{\Id}{\text{Id}}
\providecommand{\SSR}{strong subset reduces\xspace}

\providecommand{\SSNs}{strong subset reductions\xspace}

\newcommand{\SLFTJ}{singly light for the jump\xspace}
\newcommand{\dpl}{{\dotplus}}
\renewcommand{\O}{{\mathcal{O}}}

\title[On the computable  Friedman-Stanley jump]{Investigating the computable\\ Friedman-Stanley jump}

\author[U.~Andrews]{Uri Andrews}
\address{Department of Mathematics\\
University of Wisconsin\\
Madison, WI 53706-1388\\
USA}
\email{{andrews@math.wisc.edu}}
\author[L.~San Mauro]{Luca San Mauro}
\address{Department of Mathematics, Sapienza University of Rome, Italy}
\email{{luca.sanmauro@uniroma1.it}}


\makeatletter
\@namedef{subjclassname@2020}{%
	\textup{2020} Mathematics Subject Classification}
\makeatother


\begin{document}

\maketitle

\begin{abstract}
	We answer several questions about the computable Friedman-Stanley jump on equivalence relations. This jump, introduced by Clemens, Coskey, and Krakoff, deepens the natural connection between the study of computable reduction and its Borel analog studied deeply in descriptive set theory.
\end{abstract}

\section{Introduction}
Computable reducibility  of countable equivalence relations is a natural computability theoretic way of characterizing some equivalence relations as more complex than others. For equivalence relations $E$ and $R$ on the natural numbers, we say that $E$ is \emph{computably reducible} to $R$, written $E\leq R$, if there is a computable function $f$ so that $$x\mathrel{E}y \leftrightarrow f(x)\mathrel{R}f(y).$$ This notion was first introduced by Ershov \cite{ErshovPositiveEquivalences} and has seen a recent resurgence of interest \cite{ABSM, CCK,FokinaFriedman, FokinaFriedmanHarizanovEtal} with special attention paid to local structures of equivalence relations of a given complexity class such as the c.e. equivalence relations (\emph{ceers}) \cite{GaoGerdes, TheoryOfCeers, SelffullnessAndUniformJoin, JoinsAndMeets, JumpsOfCeers, SurveyOnUniversalCeers, IndexSetsOfCeers, UEICeers, IsomorphismClassesOfCeers, BadaevSorbi}  
and how they naturally arise from algebra \cite{WordProblemsAndCeers, SorbiNies,  CFMiller3, JourneyToCeStructures, CeLinearOrders, GavryushkinKhoussainovStephan}
or other levels of the arithmetical hierarchy \cite{CHM, IanovskiMillerNgNies} or levels of the Ershov hierarchy \cite{NgYu, ErshovHierarchy}. 

Part of the motivation behind studying computable reducibility is that it is a computability theoretic natural analog of the descriptive set theoretic notion of Borel reducibility, intensively studied in descriptive set theory~\cite{FriedmanStanley,kanoveui,gao}, where the equivalence relations are on standard Borel spaces and the reducing function is allowed to be Borel instead of computable.

\smallskip

In the context of computable reducibility, several notions of a jump have been studied:
\begin{defn}\label{def:jumps}
	For $E$ an equivalence relation: 
	\begin{itemize}
		\item $E'$ is defined by $x \mathrel{E'} y$ if and only if $x=y$ or $\phi_x(x)\downarrow$, $\phi_y(y)\downarrow$, and  $\phi_x(x) \mathrel{E} \phi_y(y)$.
		\item $E^{+}$ is defined by $x \mathrel{E^+} y$ if and only if $[F_x]_E = [F_y]_E$, where $F_x$ is the finite set with canonical index $x$ and for any set $S$, $[S]_E$ is the \emph{$E$-closure} of $S$, i.e., the set of elements $E$-equivalent to a member of $S$.
		\item $E^\dpl$ is defined by $x \mathrel{E^\dpl} y$ if and only if $[W_x]_E = [W_y]_E$, where $W_i$ is the $i$th c.e.\ set.
	\end{itemize}
\end{defn}

The first two jumps here were introduced by Gao and Gerdes \cite{GaoGerdes} and the last one was introduced by Clemens, Coskey, and Krakoff \cite{CCK}.
The last two here are finite and computable analogs of the Friedman-Stanley jump studied in descriptive set theory.

\begin{defn}
	For $E$ an equivalence relation on the standard Borel space $X$, the Friedman-Stanley jump $E^+$ of $E$ is the equivalence relation on $X^{\omega}$ given by 
	$$ f \mathrel{E^+} g \leftrightarrow [\ran(f)]_E=[\ran(g)]_E.$$
\end{defn}

Friedman and Stanley \cite{FriedmanStanley} showed that this jump operator is proper. That is, $E^+>_B E$ for any Borel equivalence relation $E$. Clemens, Coskey, and Krakoff \cite{CCK} showed that $E^\dpl > E$ for any hyperarithmetic (HYP in the sequel) equivalence relation. They also showed that a $\Sigma^1_1$-complete equivalence relation is a fixed point for the jump $\dpl$, i.e., $E\equiv E^\dpl$. Clemens, Coskey, and Krakoff \cite{CCK} ask several natural questions regarding features of the jump operator $\dpl$. In this paper, we answer these questions.  

Throughout the rest of this paper, the ``jump'' of an equivalence relation will always refer to the computable Friedman-Stanley jump operator $\dpl$.

\subsection{Preliminaries} We assume that the reader is familiar with the fundamental notions and techniques of computability theory.

\smallskip

 All our equivalence relations have domain the set $\omega$ of the natural numbers. Equivalence relations are  \emph{infinite}, if they have infinitely many equivalence classes; otherwise, they are \emph{finite}. For a c.e.\ set $A$, the equivalence relation $E_A$ is given by $x \mathrel{E_A} y$ if and only if $x=y$ or $x,y\in A$. A ceer of the form $E_A$ is called \emph{1-dimensional}.


The identity relation on $\omega$ is denoted by $\Id$. Note that $\Id^{\dpl}$ is equivalent is to $=^{ce}$, where $x \mathrel{=^{ce}} y$ if and only if $W_x=W_y$.
 Following \cite{JoinsAndMeets,ABSM}, we say that:
\begin{itemize}
\item $E$ is  \emph{light} if $\Id\leq E$;
\item $E$ is \emph{dark} if $E$ is infinite and not light;
\item $E$ is \emph{dark minimal} if it is dark and all equivalence relations $<E$ are finite.
\end{itemize} 

The next lemma will be used a few times.
 
 \begin{lem}[{\cite[Lemma~1.13]{ABSM}}]\label{dark minimal intersects all classes}
 Let $E$ be a dark minimal equivalence relation. If $W_e$ intersects infinitely many $R$-classes, then $W_e$ must intersect every $R$-class.
 \end{lem}

	For two equivalence relations $E,R$,
\begin{itemize}
\item the \emph{uniform join} $E\oplus R$ is the equivalence relation defined by $x \mathrel{E\oplus R} y$ if and only if $x=2k,y=2l$ and $k\mathrel{E} l$ or $x=2k+1$, $y=2l+1$ and $k\mathrel{R}l$;
\item the \emph{cross product} $E \times R$ is the equivalence relation defined by
\[ 
\langle x, y\rangle ( E \times R ) \langle u, w\rangle \Leftrightarrow
(x \mathrel{E} u  \, \wedge \, y \mathrel{R} w).
\]
\end{itemize}

For a countable sequence $(E_i)_{i\in \omega}$, $\oplus_i E_i$ is given by $\langle x,y \rangle \mathrel{\oplus_i E_i} \langle v,w\rangle$ if and only if $x=v$ and $y \mathrel{E_x} w$.

The following definition gives a convenient notation.

\begin{defn}
	For sets $X,Y$ and an equivalence relation $E$, we write $X\subseteq_E Y$ to mean $[X]_E\subseteq [Y]_E$. Similarly, we write $X=_E Y$ to mean $[X]_E= [Y]_E$ and $X\subsetneq_E Y$ to mean $[X]_E\subsetneq [Y]_E$.
	
	For any set $X$ and equivalence relation $E$, we write $X/E$ for the set of $E$-equivalence classes of members of $X$.
\end{defn}
	
\subsection{Questions of Clemens, Coskey, and Krakoff}  For every ceer $E$, we have $E^\dpl \leq \Id^\dpl$   \cite[Proposition~4.1]{CCK}. And certainly any light ceer satisfies $\Id^\dpl \leq E^\dpl$. This motivates the following definition:

\begin{defn}
	A ceer $E$ is \emph{light for the jump} if $\Id^\dpl\leq E^\dpl$. We note that this is the notion of highness for ceers using this jump operator.
\end{defn}

Clemens, Coskey, and Krakoff \cite[Question 1]{CCK} ask for a characterization of the c.e.\ sets $A$ so that $E_A$ is light for the jump.
In Section \ref{section 2: ceers which are light for the jump}, we give the following solution: 

\begin{thm}
	For a c.e.\ set $A$, $E_A^\dpl\equiv \Id^\dpl$ if and only if $A$ is not hyperhypersimple. Thus, the property of being ``light for the jump'' is $\Sigma^0_4$-complete. 
\end{thm}

This line of inquiry led us to wonder what the picture looks like for the double-jump. That is, which sets $A$ have the property that $\Id^{\dpl\dpl}\leq E_A^{\dpl\dpl}$, i.e. $E_A$ is high$_2$ for the computable FS-jump. And we also ask whether there are any ceers $E$ so that $\Id^{\dpl\dpl}\not\leq E^{\dpl\dpl}$. We answer these questions as well in Section \ref{Double Jumps of Ceers}:

\begin{thm}
	For every co-infinite c.e.\ set $A$, $\Id^{\dpl\dpl}\leq E_A^{\dpl\dpl}$. Yet there are infinite ceers $E$ so that $\Id^{\dpl\dpl}\not\leq E^{\dpl\dpl}$.
	
	In fact, every low dark minimal ceer satisfies $\Id^{\dpl\hat{k}}\not\leq E^{\dpl\hat{k}}$ (where $E^{\dpl \hat{k}}$ is the $k$th iterate of the $\dpl$-jump of $E$), yet there are dark minimal ceers $E$ so that $\Id^{\dpl\dpl}\leq E^{\dpl\dpl}$.
\end{thm}

Next, every infinite ceer $E$  has the property that $\Id\leq E^\dpl$~\cite[Theorem 4.2]{CCK}, but there are infinite $\Delta^0_4$ equivalence relations $E$ so that $\Id\not\leq E^\dpl$~\cite[Theorem 4.4]{CCK}. Clemens, Coskey, and Krakoff ask \cite[Question 6]{CCK} what is the least complexity of an infinite equivalence relation $E$ so that $\Id\not\leq E^\dpl$. In Section \ref{section: dark jumps}, we answer this with the following theorem:

\begin{thm}
	Every infinite $\Pi^0_2$ equivalence relation $E$ satisfies $\Id\leq E^\dpl$, but there are infinite $\Sigma^0_2$ equivalence relations $E$ so that $\Id\not\leq E^\dpl$.
\end{thm}

Clemens, Coskey, and Krakoff \cite{CCK} also examine the transfinite jump hierarchy, which they defined as follows:

\begin{defn}
	For $a\in \O$ and $E$ an equivalence relation, $E^{\dpl a}$ is defined by induction as follows:
	
	If $a=1$ (the notation for 0), then $E^{\dpl a}=E$.
	
	If $a=2^b$ then $E^{\dpl a} = (E^{\dpl b})^\dpl$.
	
	If $a=3\cdot 5^e$, then $E^{\dpl a} = \oplus_{i} E^{\dpl \phi_e(i)}$
\end{defn}

To avoid confusion with notations in $\mathcal{O}$, we use the following definition for finite iterates of the jump:

\begin{defn}
	For $n\in \omega$ and $E$ an equivalence relation, we let $E^{\dpl \hat{n}}$ be the $n$th iterate of the jump over $E$.
\end{defn}

 Clemens, Coskey, and Krakoff show \cite[Theorem 3.1]{CCK} that no jump fixed-point can be hyperarithmetic (HYP). In fact, they show that if $E$ is a jump fixed point and $X$ is a HYP set, then $X\leq_m E$ \cite[Theorem 3.10]{CCK}. They ask if notations matter in the definition of the jump \cite[Question 2]{CCK} and if every jump fixed point must be an upper bound under computable reduction (not just $m$-reduction) for all HYP equivalence relations \cite[Question 3]{CCK}. We answer both in  the affirmative in Section \ref{section: jumps depend on notations} and \ref{section: every HYP ER reduces to some Id^+a}:

\begin{thm}
	There are two notations $a,b\in \O$ with $\abs{a}=\abs{b}=\omega^2$ so that $\Id^{\dpl a}$ and $\Id^{\dpl b}$ are incomparable. 
\end{thm}


	On the other hand, if $\abs{a}=\abs{b}$, then $\Id^{\dpl a}$ and $\Id^{\dpl b}$ are somewhat related as follows:
	
	\begin{thm}
		For every computable ordinal $\alpha$, there is an equivalence relation $E$ which is $\Pi^0_{2\cdot \alpha+1}$ so that whenever $a\in \mathcal{O}$ is a notation for $\alpha$, we have $\Id^{\dpl a}\leq E$.
	\end{thm}

\begin{thm}
	For every HYP equivalence relation $E$, there is a notation $a\in \O$ so that $E\leq \Id^{\dpl a}$. In particular, if $E$ is a fixed point of the jump, i.e., $E\equiv E^\dpl$ then $E$ is an upper bound for every HYP equivalence relation.
\end{thm}

\section{Ceers which are light for the jump}\label{section 2: ceers which are light for the jump}
In this section, we examine which ceers $E$ are light for jump, i.e., $\Id^\dpl \leq E^\dpl$. We begin by introducing a purely combinatorial notion which will capture a ceer being light for the jump.

\begin{defn}
	A ceer $E$ is \emph{\SLFTJ} if there is a uniformly c.e.\ sequence $(V_i)_{i\in \omega}$ so that, $V_i\not\subseteq_E \bigcup_{j\neq i} V_j$ for every $i\in \omega$. That is, there is an $x\in V_i$ so that $[x]_E\cap V_j=\emptyset$ for every $j\neq i$.

\end{defn} 

This definition naturally captures a ceer being light for the jump in a way given by a map from $\omega$ into c.e.\ sets.

\begin{lem}\label{Motivating SLFTJ}
	Fix a ceer $E$. Then, $E$ is \SLFTJ if and only if there exists a function $f$ so that the map $g$ which sends $i$ to an index for $\bigcup_{j\in W_i}W_{f(j)}$ gives a reduction of $\Id^\dpl$ to $E^\dpl$. 
\end{lem}
\begin{proof}
	Suppose first that $E$ is \SLFTJ. We let $W_{f(j)}=V_j$. Since $V_j$ contains an element whose $E$-class is not intersected by any $V_k$ with $k\neq j$, the image  $\bigcup_{j\in W_i}W_{f(j)}$ of a c.e.\ set $W_i$ determines whether $j\in W_i$. Thus, this gives a reduction of $\Id^\dpl$ to $E^\dpl$.
	
	Next, suppose that there is a function $f$ as given. If every element of $W_{f(j)}$ were to be $E$-equivalent to a member of $W_{f(k)}$ for some $k\neq j$, then the $g$-image of $\omega$ and $\omega \smallsetminus \{j\}$ would be the same, so $g$ would not be a reduction of $\Id^\dpl$ to $E^\dpl$. Thus the family $V_j=W_{f(j)}$ shows that $E$ is \SLFTJ.
\end{proof}

More surprisingly, we show that any ceer which is light for the jump is \SLFTJ. Before this, let us establish a useful lemma that constrains the behavior of any reduction  from $\Id^{\dpl}$ to some $E^{\dpl}$.

	\begin{lem}\label{properties of h for jumps}
		Let $h: \Id^{\dpl} \leq E^{\dpl}$, for a ceer $E$. The following hold:
		\begin{enumerate}
			\item\label{h respects inclusion} if $W_i\subseteq W_j$, then $W_{h(i)}\subseteq_E W_{h(j)}$;
			\item \label{lem:Light For The Jump Is Finite Basedd} if $W_i$ is  infinite, then $ W_{h(i)}=_E\bigcup_{W_a\subset W_i \text{  finite}} W_{h(a)}.$
		\end{enumerate}
	\end{lem}
	\begin{proof}
		$(1)$: Suppose $x\in [W_{h(i)}]_E\smallsetminus [W_{h(j)}]_E$. Then, we use an index $e$ we control by the recursion theorem and we let $W_e=W_i$ unless $x\in [W_{h(j)}]_E$, in which case we make $W_e=W_j$. This gives a contradiction.
		\smallskip
		
		$(2)$: We already have $\bigcup_{W_a\subset W_i \text{  finite}} W_{h(a)}\subseteq_E W_{h(i)}$ by the first item. Suppose that $y \in [W_{h(i)}]_E$. Then, we use an index $e$ we control by the recursion theorem 
		and we begin enumerating $W_{h(i)}$ into $W_e$ until we see $y \in [W_{h(e)}]_E$. At this point, we stop enumerating any new elements into $W_e$.  We thus get a finite set $W_e$ and $y\in [W_{h(e)}]_E$.
	\end{proof}


\begin{thm}
	A ceer is light for the jump if and only if it is \SLFTJ.
\end{thm}
\begin{proof}
	If $E$ is \SLFTJ, then it is light for the jump by Lemma~\ref{Motivating SLFTJ}.
	
	Let $E$ be a ceer which is light for the jump and fix $h$ to be a reduction of $\Id^\dpl$ to $E^\dpl$. We will construct a sequence $(V_i)_{i\in \omega}$ of c.e.\ sets witnessing that $E$ is \SLFTJ.
	
	We define a function from c.e.\ sets $F$ to c.e.\ sets $W(F)$ by taking an index $e$ we control by the recursion theorem and enumerating $F$ into $W_e$. Then we let $W(F)=W_{h(e)}$. At a given stage $s$, we let $W(F)_s=W_{h(e),s}$. We observe that for any index $i$ of $F$, we have $W(F)=W_{h(i)}$.
%
Moreover, by Lemma~\ref{properties of h for jumps}, we may assume that for every $s$ we have $W(F)_s\subseteq W(G)_s$ for any finite sets $F\subseteq G$. 

We fix a sequence of equivalence relations $E_s$ which limit to $E$ and we assume that at most one pair of classes collapses at any given stage $s$. Our construction is designed to meet the following requirements:

	\begin{itemize}
\item[$\mathcal{P}_i$]: $(\exists x \in V_i)(x$ is not $E$-equivalent to any $y\in V_j$, for $j\neq i$).
	\end{itemize}

\subsection*{Strategy}
Intuitively, the strategy to satisfy $\mathcal{P}_i$ acts as follows: We choose a number $a_i$ and begin with a set $B_i=\emptyset$. We want to exploit the fact that $W(B_i)\subsetneq_E W(B_i\cup \{a_i\}$). So, we choose a number $z$ which we believe is in $W(B_i\cup \{a_i\})\smallsetminus [W(B_i)]_E$ and we put this $z$ into $V_i$. If we  see $z$ become $E$-equivalent to a member of set $V_j$ with $j>i$, then this is because some set $B_j\cup \{a_j\}$ which does not contain $a_i$ has $z\in [W(B_j\cup \{a_j\})]_E$. We now give up on $z$ and update our parameter $B_i$ to contain $B_j\cup \{a_j\}$ and try to use the fact that $W(B_i)\subsetneq_E W(B_i\cup \{a_i\})$ for this new larger set $B_i$, and we choose a new number $z$. If this happens infinitely often, and each choice of $z$ ends up in $\bigcup_{j>i} [V_j]_E$, then we will have built a set $B_i$ not containing $a_i$ so that $W(B_i)=W(B_i\cup \{a_i\})$ contradicting the fact that $h$ is a reduction of $\Id^{\dpl}$ to $E^{\dpl}$. 

If we see $z$ go into $\bigcup_{j<i}[V_j]_E$, it is possible that this $E$-class is the only one distinguishing between $[W(B_i)]_E$ and $[W(B_i\cup \{a_i\})]_E$. So, we put $a_i$ into $B_i$ and choose a new parameter $a_i$. Now this class is already in $W(B_i)$, and since $\bigcup_{j<i} V_j$ will be finite, we will have to do this only finitely often, so the above strategy will eventually find us a $z\in V_i\smallsetminus \bigcup_{j\neq i} [V_j]_E$.

	\subsection*{Construction}
	
	The strategy for the $\mathcal{P}_i$ requirement will have parameters $a_i$, $B_i$ and $z_i$. These should be understood as follows: $B_i$ is a finite set which does not contain $a_i$. We want to use the fact that $W(B_i)\neq W(B_i\cup \{a_i\})$ to find an $E$-class which ``represents'' $a_i$. The parameter $z_i$ defines an element which is in $V_i\smallsetminus [\bigcup_{j\neq i} V_j]_E$ at the current stage. To refer to the value of a parameter at the end of stage $s$, we give it a superscript $s$.
	
	The strategy for $\mathcal{P}_i$ requires attention at stage $s+1$ if its parameter $z_i$ is undefined or is contained in $[\bigcup_{j\neq i} V_{j,s}]_{E_{s}}$
	or if it has been injured since it last acted. At any given stage, the highest priority strategy which requires attention acts. All lower-priority strategies are injured. For bookkeeping reasons, if a strategy is injured, it keeps its parameters but just knows that it is injured.
	The strategy acts as follows when acting at stage $s+1$:

	\subsubsection*{Step (I)} If the strategy has been injured since it last acted or if it has never acted before, then it chooses new parameters as follows: If it currently has parameters $a_i^s$ and $B_i^s$ defined, then it lets the parameter $B_i$ have value $B_i^s\cup \{a_i^s\}$. Otherwise, it lets the parameter $B_i$ have value $\emptyset$. It also chooses a new parameter $a_i$ to be a fresh number which has never before been considered.
	
%
%
%
%

	\subsubsection*{Step (II)} If $z_i$ is currently defined we run the module \texttt{TryTheNumber}($z_i$). Otherwise, we run the \texttt{PickANumber} module.
	
\smallskip	
	
	We now describe the module
	\texttt{TryTheNumber}($c$):
	\begin{enumerate}
	\item If $c \mathrel{\cancel{E_s}} w$ for every $w\in \bigcup_{j\neq i} V_{j,s}$, then we let $z_i^{s+1}=c$ and enumerate $c$ into $V_i$.

	\item If $c \mathrel{E_s} w$ for some $w\in V_{j,s}$ with $j<i$, then we let $D=B_i\cup \{a_i\}$ and we pick a new number $b$. We then reset the parameters $B_i=D$ and $a_i=b$. We then call the \texttt{PickANumber} module with these new parameters.
	
	\item If $c \mathrel{E_s} w$ for some $w\in V_j$ with $j>i$, then we let $D=B_i\cup \{a_j\}\cup B_j$. We reset $B_i$ to be $D$ and we call the \texttt{PickANumber} module with the new parameters (note that $a_i$ has not changed). 
	\end{enumerate}

We now describe the 
	\texttt{PickANumber} module: 
\begin{itemize}	
	\item[] Find the first $t>s$ so that $[W(B_i\cup \{a_i\})_t]_{E_s}\neq [W(B_i)_t]_{E_s}$ and let $c$ be the least element of $W(B_{i}\cup \{a_i\})_t\smallsetminus [W(B_i)_t]_{E_s}$. 
We then call the module \texttt{TryTheNumber}($c$).
\end{itemize}

\subsection*{Verification}

Note that we only ever enumerate a number into $V_i$ if it is already in $W(B_i\cup \{a_i\})$ and we only ever grow the set $B_i\cup \{a_i\}$ (either by putting $a_i$ into $B_i$ in case the strategy is injured or in case 2 of the \texttt{TryTheNumber} module, or by keeping $a_i$ the same and growing $B_i$ in case 3 of the \texttt{TryTheNumber} module), so we always have $V_{i,s}\subseteq {W(B_i^s\cup \{a_i^s\})}_s$.

\begin{lem}
	If the strategy $\mathcal{P}_i$ begins the \emph{\texttt{PickANumber}} module, it eventually terminates in case $(1)$.
\end{lem}
\begin{proof}
	It suffices to see that the stategy cannot take outcome (2) or (3) of the \texttt{TryTheNumber} module infinitely many times. Every time it takes outcome (2) or (3), we have a new element $w\in \bigcup_{j\neq i} V_{j,s}$ so that $w\in [W(B_i)]_{E_s}$. Note that $c$ was in $W(B_i\cup \{a_i\})_t\smallsetminus [W(B_i)_t]_{E_s}$ before the change of parameters, but $c\in [W(B_i)_t]_{E_s}$ after the change of parameters. Since $c\mathrel{E_s} w$, we also see $w$ has entered the set $[W(B_i)_t]_{E_s}$. Note that since we only ever grow $B_i$, once something is seen to be in $W(B_i)$, it remains there. Since $\bigcup_{j\neq i} V_{j,s}$ is finite at a given stage of the construction, this process must eventually stop.
\end{proof}

\begin{lem}\label{a_i never enters B_i}
	At every stage $s$, if $i<j$ and $a_i$ is defined, then $a_i\notin B_j\cup \{a_j\}$.
\end{lem}
\begin{proof}
	This is by induction on stages. When $a_i$ is chosen, it is chosen new so this holds at that stage. Similarly, $a_j$ is chosen new so $a_i\neq a_j$. At later stages, elements can enter $B_j$ by either adding $a_j$ to $B_j$ in outcome (2) of the \texttt{TryTheNumber} module or by adding $\{a_k\}\cup B_k$ for some $k>j$ into $B_j$. But $a_i\notin \{a_k\}\cup B_k$ by the inductive hypothesis.
\end{proof}

\begin{lem}
	Every strategy eventually settles with a parameter $z_i\notin [\bigcup_{j\neq i} V_j]_E$. Thus, every $V_i$ is finite and contains an element which is not $E$-equivalent to a member of any other $V_j$.
\end{lem}
\begin{proof}
	We proceed by induction. We may assume that every strategy $\mathcal{P}_j$ for $j<i$ has found such parameters $z_i$ by stage $s$. Since these parameters never change after stage $s$, those strategies never act after stage $s$ and the $\mathcal{P}_i$-strategy is never injured after stage $s$. The $\mathcal{P}_i$-strategy can then only take outcome (2) of the \texttt{TryTheNumber} module finitely often as there are only finitely many members of $V_j$ for $j<i$.
	
	Let $t>s$ be a stage late enough that the $\mathcal{P}_i$-strategy never takes outcome (2) of the \texttt{TryTheNumber} module after stage $t$. Then the parameter $a_i$ at stage $t$ is permanent. Further, note that $a_i$ never enters $B_i$ after stage $t$. This cannot happen via outcome (2), since outcome (2) never happens after stage $t$ and $a_i$ never enters $B_i$ via outcome (3) by Lemma~\ref{a_i never enters B_i}. 
	
	Considering the limiting value of $B_i$, since $a_i\notin B_i$, we see that $[W(B_i)]_E\subsetneq [W(B_i\cup \{a_i\})]_E$. Let $c$ be the least element of $W(B_i\cup \{a_i\})\smallsetminus [W(B_i)]_E$ and let $u>t$ be a stage large enough that $[W(B_i^u)_u]_{E_u}\cap [0,c] = [W(B_i)]_E\cap [0,c]$ and $[W(B_i^u\cup \{a_i^u\})_u]_{E_u}\cap [0,c] = [W(B_i\cup \{a_i\})]_E\cap [0,c]$. Then when we next run the \texttt{PickANumber} module after stage $u$, we pick this value of $c$ and we cannot take outcome (2) of \texttt{TryTheNumber}($c$) because $u>t$ and we cannot take outcome (3) as this would put $c$ into $[W(B_i)]_E$. Thus we must take outcome 1 so $c\in V_i$. 
	
	Now we argue that $c\notin [\bigcup_{j\neq i} V_j]_E$. Suppose towards a contradiction that $c\mathrel{E} w $ for $w\in V_j$ with $j\neq i$. Then the $\mathcal{P}_i$ strategy requires attention and since every higher priority strategy has settled, it gets to act. It then runs the \texttt{TryTheNumber}($c$) module and must take outcome (2) or (3) depending on whether $j<i$ or $j>i$. This cannot take outcome (2) as $u>t$. If it takes outcome (3), then we see $c\in [W(B_i)]_E$ contradicting the choice of $c$.
\end{proof}
This concludes the proof that the property of being light for the jump coincides with the property of being singly light for the jump.
\end{proof}

We now shift the focus to the case of $1$-dimensional ceers. Indeed, it is natural to ask for which c.e.\ sets $A$ is $E_A$ light for the jump (\cite[Question 1]{CCK}).
Clemens, Coskey, and Krakoff proved the following: on the one hand, if $A$ is not hyperhypersimple then $E_A$ is light for the jump \cite[Theorem 4.8]{CCK}; on the other hand, if $A$ is quasi-maximal, then $E_A$ is not light for the jump \cite[Theorem 4.17]{CCK}. This is not a characterization, as there are sets which are hyperhypersimple yet are not quasimaximal~\cite{Robinson2Theorems}. But the next corollary settles the problem.

\begin{cor}\label{cor: light for jump nonhyperhypersimple}
	For any c.e.\ set $A$, $E_A$ is light for the jump if and only if $A$ is nonhyperhypersimple.
\end{cor}
\begin{proof}
	If $A$ is nonhyperhypersimple, then $E_A$ is light for the jump by \cite[Theorem 4.8]{CCK}.
	
	Suppose $E_A$ is light for the jump. Then $E_A$ is \SLFTJ. Let $\mathcal{V}=(V_i)_{i\in \omega}$ witness this. We may assume that every $V_i$ has an element $z_i$ which is not in $\bigcup_{j\neq i} V_j$ and $z_i\notin A$. This is because all of $A$ constitutes a single class in $E_A$, so omitting one set from the sequence of $V_i$ suffices to guarantee this. We may also assume that at every stage at most one number is enumerated into at most one set $V_i$.
	
	We now define the sets $X_i$ defined as follows: $z\in X_i$ if 
	\begin{enumerate}
		\item $z\in V_i$ and $V_i$ is the first set in $\mathcal{V}$ which $z$ enters.
		\item $(\exists s>z \forall w<z )(w\in V_{i,s}\rightarrow w\in A_s\cup \bigcup_{j\neq i}V_{j,s})$.
	\end{enumerate}

The first condition implies that $\mathcal{X}=(X_i)_{i\in \omega}$ is a uniformly c.e.\ array of disjoint sets. Since every $V_i$ contains a member which is not contained in $A\cup \bigcup_{j\neq i} V_j$, the second condition ensures that each $X_i$ is finite. Finally, for each $V_i$, let $z$ be the least member of $V_i\smallsetminus (A\cup \bigcup_{j\neq i}V_j)$. Then $z\in X_i$. Thus $\mathcal{X}$ witnesses that $A$ is not hyperhypersimple.
\end{proof}

\begin{cor}
	The index set of ceers which are light for the jump is $\Sigma^0_4$-complete.
\end{cor}
\begin{proof}
It is easy to calculate that being light for the jump is a $\Sigma^0_4$ problem. To conclude, it is sufficient to recall that the index set of nonhyperhypersimple c.e.\ sets is $\Sigma^0_4$-complete (see \cite{yates1,yates2}, where the result is announced, and  \cite[Theorem XII 4.13]{Soare} where it is proved) and then use Corollary~\ref{cor: light for jump nonhyperhypersimple}.
\end{proof}

We finish our discussion of which ceers  are light for the jump by focusing on a special class of ceers which will also be  considered        in the next section: dark minimal ceers, i.e., dark ceers $E$ so that $R<E$ implies that $R$ is finite. Dark minimal ceers are of special interest for the theory of ceers. For example, we code graphs onto the dark minimal ceers to show that the theory of the partial order of ceers is as rich as possible, being computably isomorphic with first-order arithmetic~\cite{TheoryOfCeers}.

\begin{prop}
No dark minimal ceer is light for the jump.
\end{prop}

\begin{proof}

Towards a contradiction, suppose that there exists $h: \Id^\dpl\leq E^\dpl$, for a dark minimal ceer $E$.  Lemma~\ref{properties of h for jumps}(2) guarantees that, if $W_i$ is infinite, then $W_{h(i)}$ must intersect infinitely many $E$-classes, as otherwise there would be finite c.e.\ sets $W_a \subset W_b \subset W_i$ so that $[W_a]_E=[W_b]_E$, a contradiction. So, let $W_{e_0}$ and $W_{e_1}$ be the evens and the odds, respectively. Since $E$ is dark minimal, by Lemma~\ref{dark minimal intersects all classes}, we  obtain that  $[W_{h(i)}]_E= [W_{h(j)}]_E=\omega$, a contradiction.
%
\end{proof}


\section{The Higher Jump Hierarchy of Ceers}\label{Double Jumps of Ceers}

We now turn our attention to higher jumps applied to ceers. We first consider the 1-dimensional case where, contrary to the picture for the single jump, every co-infinite c.e.\ set $A$ has the property that $E_A$ has the highest possible double-jump. Of course, we focus on the co-infinite c.e.\ sets because, if $A$ is co-finite, then $E_A$ has only finitely many classes.

\begin{thm}
	If $A$ is a co-infinite c.e.\ set, then $\Id^{\dpl\dpl}\leq E_A^{\dpl\dpl}$.
\end{thm}
\begin{proof}
	We describe an algorithm $h$ for reducing $\Id^{\dpl\dpl}$ to $E_A^{\dpl\dpl}$. Let $F:\omega\rightarrow \omega$ be so $F(n)$ is the $n$th element of $\omega\smallsetminus A$. Note that $F$ is $\Delta^0_2$, so we fix also $F_s$ a uniformly computable sequence of functions limiting to $F$.
	
	We arrange it so that for any index $e$, $h(e)$ is an index for a uniformly c.e.\ family consisting of $\omega$, all finite sets, and $\omega\smallsetminus \{F(k)\}$ for each $k$ so that $W_k=W_{i}$ for some $i\in W_e$. We observe that this is a reduction from $\Id^{\dpl\dpl}$ into $E^{\dpl\dpl}$.
	
	Fix an index $e$ and we must uniformly produce the uniform family which is to be its image under $h$.
	Begin with a uniform enumeration of $\omega$ and all finite sets. We add to this a sequence of sets $V_{i,k}^m$. If $i$ enters $W_e$, then make $V_{i,k}^0$ active. If $V_{i,k}^j$ is active for some $j$ and $F_{s+1}(k)\neq F_s(k)$, then we deactivate $V_{i,k}^j$, make $V_{i,k}^j=\omega$ and we activate $V_{i,k}^{j+1}$.
		If $V_{i,k}^j$ is active at stage $s$ and both $s$ and the length of agreement between $W_i$ and $W_k$ at stage $s$ are $\geq \ell$, then we enumerate $[0,\ell]\smallsetminus \{F_s(k)\}$ into $V_{i,k}^j$.
		
	
	It is straightforward to check that this gives a uniform enumeration of the described family.	
\end{proof}

Next we see that, unlike the 1-dimensional case, there are ceers which are not high$_n$ for the computable Friedman-Stanley jump for any $n$. That is, $\Id^{\dpl\hat{n}}\not\leq E^{\dpl\hat{n}}$. We do this by considering the low dark minimal ceers. Dark minimal ceers have been used heavily in the literature, and we now note that there are dark minimal ceers which are also low\footnote{We emphasize that we are using lowness in the sense of the Turing jump on sets, not any of the equivalence relation jumps from Definition \ref{def:jumps}}      .

\begin{lem}
	There are low dark minimal ceers.
\end{lem}
\begin{proof}
The construction of a dark minimal ceer $E$ has requirements of two types:
\begin{itemize}
	\item[$\mathcal{R}_{e,n}$:] If $W_e$ is intersects infinitely many $E$-classes, then it intersects $[n]_E$.
	\item[$\mathcal{I}_m$:] $E$ has $\geq m$ classes.
\end{itemize}	

To these, we can add the lowness requirement: 
\begin{itemize}
\item[$\mathcal{L}_e$:] If for infinitely many stages $s$ we have $\phi^{E_s}_e(e)\downarrow$, then $\phi^E_e(e)$ converges.
\end{itemize}
$\mathcal{L}$-requirements only place restraint on some finite collection of $E$-classes preventing collapse. This fits in the finite injury construction of a dark minimal ceer, as given in \cite{JoinsAndMeets} (i.e., to a lower-priority requirement, this restraint is no different than the restraints placed by higher-priority $\mathcal{I}$-requirements).
\end{proof}

Recall that all dark minimal ceers $E$ have the property that if $W_e$ intersects infinitely many $E$-classes, then $W_e$ must intersect every $E$-class.
The following few lemmas use this property to bound the complexity of the jumps of dark minimal ceers.

\begin{lem}\label{lem:lowness makes stuff Delta2}
	If $E$ is a dark minimal ceer, then for each $k\in \omega$, the set of $i$ so that $W_i/E$ has size $\geq k$ is a $\Delta^0_2(E)$ set.
	
	Further, the set of triples $(i,j,k)$ so that $\vert W_i/E\vert = k$ and $W_i=_E W_j$ is $\Delta^0_2(E)$.
	
	In particular, if $E$ is a low dark minimal ceer then these sets are both $\Delta^0_2$.
\end{lem}
\begin{proof}
	The quotient $W_i/E$ has size at least $k$ if and only if $\exists x_1,\ldots x_k\in W_i \bigwedge_{k\neq j} {x_k\mathrel{\cancel E} x_j}$. This is $\Sigma^0_1(E)$.
	
	To check if $(i,j,k)$ is so that $\vert W_i/E\vert = k$ and $W_i=_E W_j$, we can in a $\Delta^0_2(E)$ way check that $\vert W_i/E\vert = k$ and $\vert W_j/E\vert = k$ by the above. Then, if this is the case, we can in a $E$-computable way find elements $x_1,\ldots x_k\in W_i$ so that $\bigwedge x_i \mathrel{\cancel E} x_j$ and $y_1,\ldots y_k\in W_i$ so that $\bigwedge y_i \mathrel{\cancel E} y_j$. Then we need only check in a $E$-computable way that $\bigwedge_{i\leq k} x_i \mathrel{E} y_{\sigma(i)}$ for some permutation $\sigma$.
\end{proof}

\begin{lem}
	If $E$ is a dark minimal ceer, then $E^{\dpl\dpl}$ is $\Delta^0_4(E)$.
	
	In particular, if $E$ is a low dark minimal ceer then $E^{\dpl\dpl}$ is $\Delta^0_4$. 
\end{lem}
\begin{proof}
	Let $\mathcal{V}_i,\mathcal{V}_j$ be two uniformly c.e.\ families of c.e.\ sets (given by appropriate indices, i.e., $\mathcal{V}_i = \{W_m \colon m\in W_i\})$. Then $W_i \subset_{E^{\dpl}} W_j$ if and only if  the following hold:

	\begin{equation}\label{eq1}
	(\forall S\in \mathcal{V}_i)(\forall k\in \omega) \left[
	\vert S/E\vert =k \rightarrow (\exists F\in \mathcal{V}_j) \left(
	F=_E S
	\right)
	\right]
	\end{equation}	
	\begin{equation}\label{eq2}
	(\exists S\in \mathcal{V}_i)(\forall k\in\omega)
\left[	
	 \vert S/E\vert >k \rightarrow (\exists S\in \mathcal{V}_j)(\forall k\in\omega)(\vert S/E\vert >k)
	 \right].
	\end{equation}
	
	The conditions $\vert S/E\vert = k$ and $F=_E S$ in \eqref{eq1} are $\Delta^0_2(E)$ by Lemma~\ref{lem:lowness makes stuff Delta2}. Thus, the condition \eqref{eq1} is $\Pi^0_3(E)$. Similarly, using Lemma~\ref{lem:lowness makes stuff Delta2}, \eqref{eq2} is $\Delta^0_4(E)$. Thus, $W_i=_{E^\dpl} W_j$, or $i \mathrel{E^{\dpl\dpl}} j$ is a $\Delta^0_4(E)$ condition.
\end{proof}

\begin{cor}\label{cor: complexity of jumps of dark minimals}
	If $E$ is a dark minimal ceer, then for any $k>2$, the equivalence relation $E^{\dpl \hat k}$
	is $\Pi^0_{2k-1}(E)$. In particular, if $E$ is a low dark minimal ceer then $E^{\dpl \hat k}$ is $\Pi^0_{2k-1}$.
\end{cor}
\begin{proof}
	This is by induction with base case $k=3$:
	$E^{\dpl\dpl\dpl}$ is $\Pi^0_2$ over $E^{\dpl\dpl}$, which is $\Delta^0_4(E)$, so is $\Pi^0_5(E)$. Then $E^{\dpl \widehat{(k+1)}}=(E^{\dpl \hat k})^{\dpl}$ is $\Pi^0_2$ over $E^{\dpl \hat k}$ which is $\Pi^0_{2k-1}(E)$ by induction, so $E^{\dpl \widehat{(k+1)}}$ is $\Pi^0_{2(k+1)-1}(E)$.
\end{proof}

Below, in Corollary~\ref{Jumps of Id aren't simple}, we will show that $\Id^{\dpl \hat n}$ is not $\Pi^0_{2n-1}$. It follows from this that if $E$ is a low dark minimal ceer and $k>2$, then $\Id^{\dpl \hat k}\not\leq E^{\dpl \hat k}$. Thus we will have the following theorem.

\begin{thm}
	If $E$ is a low dark minimal ceer, then $E$ is not high$_{n}$ for the computable Friedman-Stanley jump for any $n\in \omega$. 
\end{thm}
%

We now see that the assumption of lowness is necessary here, since there are dark minimal ceers so that $\Id^{\dpl\dpl}\leq E^{\dpl\dpl}$.



\begin{thm}\label{dark minimal ceer high2}
	There is a dark minimal ceer $E$ so that $\Id^{\dpl\dpl}\leq E^{\dpl\dpl}$.
\end{thm}
\begin{proof}
	Along with the ceer $E$, we construct uniformly in each $j,k,\bar{x}\in \omega$, a finite sequence of c.e.\ sets $U^n_{j,k,\bar{x}}$ for $n\leq N(j,k,\bar{x})$ and satisfying the following:
	
\begin{itemize}	
	\item[$\mathcal{Q}_{j,k,\bar{x}}$]: For every $n<N(j,k,\bar{x})$, $U^n_{j,k,\bar{x}}=\omega$:

\medskip
	
\begin{itemize}
\item If $\bar{x}$ is a $2k$-tuple which is $E$-distinct and $W_j=W_k$, then
\[
U^{N(j,k,\bar{x})}_{j,k,\bar{x}}=[\bar{x}]_E.
\]
\item Otherwise, $\vert U^{N(j,k,\bar{x})}_{j,k,\bar{x}}/E\vert$ is odd or $U^{N(j,k,\bar{x})}_{j,k,\bar{x}}=\omega$.
\end{itemize}	
\end{itemize}
	
	From the success of these requirements, we give a reduction of $\Id^{\dpl\dpl}$ to $E^{\dpl\dpl}$. Given a uniformly c.e. family $\mathcal{V}_i=\{W_j \colon j\in W_i\}$, we map this to a family $\mathcal{F}_i$ which contains each set $U^n_{j,k,\bar{x}}$ for each $j\in W_i$, $k\in \omega$ and $\bar{x}\in \omega^{2k}$, and $n\leq N(j,k,\bar{x})$. We also include an enumeration of $\omega$ and sets $X_{\bar{x}}$ for every $\bar{x}$ of odd size where $X_{\bar{x}}$ enumerates $[\bar{x}]_E$ unless we see that $\bar{x}$ is not $E$-distinct, in which case $X_{\bar{x}}$ enumerates $\omega$. It is easy to check that the sets enumerated as $X_{\bar{x}}$ are exactly $\omega$ and every $E$-closed set $Y$ so that $\vert Y/E\vert$ is odd. Further, if $W_k$ is represented in $\mathcal{V}_i$, then there is some $j\in W_i$ so that $W_j=W_k$. In this case, $U^{N(j,k,\bar{x})}_{j,k,\bar{x}}$ for various $\bar{x}$ will enumerate every $E$-closed set $Y$ so that $\vert Y/E\vert$ has size $2k$. So, this gives the necessary reduction to witness that $\Id^{\dpl\dpl}\leq E^{\dpl\dpl}$.

	It remains to verify that we can construct a dark minimal ceer $E$ along with the uniform sequence of sets $U^n_{j,k,\bar{x}}$ satisfying the $\mathcal{Q}$-requirements. 
	
	
	We have the full set of requirements for $m,n,o,j,k\in \omega$ and $\bar{x}\in \omega^{2k}$
	
	\begin{itemize}
	\item[$\mathcal{I}_m$]: $E$ has at least $m$ equivalence classes.
	\item[$\mathcal{P}_{n,o}$]: If $W_n$ intersects infinitely many $E$-classes, then $W_n$ intersects $[o]_E$.
	\item[$\mathcal{Q}_{j,k,\bar{x}}$]: Enumerate a c.e.\ set $U$ so that:
	\begin{itemize}
		\item If $\bar{x}$ is not $E$-distinct, then $U=\omega$.
		\item If $\bar{x}$ is $E$-distinct and $W_j=W_k$, then $U=[\bar{x}]_E$.
		\item If $\bar{x}$ is $E$-distinct and $W_j\neq W_k$, then $\vert U/E\vert$ is odd.
	\end{itemize}
	\end{itemize}

We note that whenever a $\mathcal{Q}_{j,k,\bar{x}}$-requirement is reinitialized, we will let the constructed set be $\omega$ and have the strategy begin constructing a new set $U$. This explains the finite sequence of sets $U_{j,k,\bar{x}}^n$ and $N(j,k,\bar{x})$ will be the number of times this strategy is reinitialized.

We enumerate the strategies in order type $\omega$. Whenever a $\mathcal{P}$-strategy causes collapse, all lower-priority strategies are reinitialized.
	
	

	The strategies for $\mathcal{I}$ and $\mathcal{P}$-requirements are familiar from the usual construction of a dark minimal set: $\mathcal{I}$-requirements simply choose a new tuple and place restraint. 
%

	$\mathcal{P}_{n,o}$-strategies seek to find an element of $W_n$ which is not (currently) $E$-equivalent to any restrained number. Then it $E$-collapses this number with $o$.
	
	$\mathcal{Q}_{j,k,\bar{x}}$ strategies act as follows: If it ever sees some $x_i \mathrel{E} x_j$, then it just stops and makes $U=\omega$, and the requirement is satisfied. Nonetheless, the strategy restrains the tuple $\bar{x}$. We begin by enumerating $[\bar{x}]_E$ into $U$. We use the $\Pi^0_2$ approximation to the statement $W_j=W_k$. That is, at every stage, we have a computable guess as to whether or not $W_j=W_k$. If we infinitely often guess that $W_j=W_k$, then they are equal. When our guess switches from saying $W_j=W_k$ to saying that they are not equal, we take a new number $y$, and we add $y$ to $U$. Further, we place restraint on the number $y$ so that lower priority requirements will not collapse $y$ with any element of $\bar{x}$. If we later guess that $W_j=W_k$, then we collapse $y$ with $x_0$. We then undefine the parameter $y$ and unrestrain it (it is restrained automatically anyway by our restraint on $x_0$).
	
	
	

The construction is put together via standard finite injury machinery. At every stage $s$, the first $s$ strategies get to act in order.

\begin{lem}
	At every moment of the construction, the set of parameters of $y$ for various $\mathcal{Q}$-requirements and the set of restrained elements for $\mathcal{I}$-requirements are all $E$-distinct.
	
	At every moment of the construction, if $\mathcal{Q}_{i,j,\bar{x}}$ is higher priority than $\mathcal{Q}_{i',j',\bar{x}'}$, then the latter's parameter $y'$ (if defined) is not $E$-equivalent to any $x\in \bar{x}$.
\end{lem}
\begin{proof}
	These statements are preserved by the choice of parameters, since they are chosen new.
	Collapse occurs only via action from $\mathcal{P}$ or $\mathcal{Q}$-requirements. In the former case, $\mathcal{P}_{n,o}$ collapses some member $z$ of $W_n$ to $o$. This $z$ was not equivalent to any element restrained by a higher-priority requirement, and since all lower-priority requirements are reinitialized, we have added no restrained number to the class of $o$. Next we consider collapse caused by a $\mathcal{Q}_{j,k,\bar{x}}$-strategy. Since $\mathcal{Q}_{j,k,\bar{x}}$ previously restrained $y$, the inductive hypothesis shows that no other parameter $y'$ for a $\mathcal{Q}'$-requirement or an element restrained by an $\mathcal{I}$-requirement was equivalent to $y$. Since after the collapse of $y$ with $x_0$, this $y$ is no longer the parameter for $\mathcal{Q}_{j,k,\bar{x}}$, we have added no such element to the class of $x_0$. Thus the first statement is proved.
	
	It remains to see that a collapse caused by a $\mathcal{Q}_{i^0,k^0,\bar{x}^0}$-strategy does not cause a violation of the second statement. By the first statement, no two $y$-parameters could have been equivalent. So, the only way this could have caused the violation is if $x^0_0\mathrel{E}_s y'$ and $y^0\mathrel{E}_s x_0$. But by inductive hypothesis, the former implies $\mathcal{Q}_{i^0,k^0,\bar{x}^0}$ is lower priority than $\mathcal{Q}_{i',j',\bar{x}'}$ and the latter implies $\mathcal{Q}_{i,j,\bar{x}}$ is priority than $\mathcal{Q}_{i^0,j^0,\bar{x}^0}$. Thus we would have $\mathcal{Q}_{i,j,\bar{x}}$ being lower priority than $\mathcal{Q}_{i',j',\bar{x}'}$, so this is not a violation of the second statement after all.
\end{proof}

\begin{lem}
	Every strategy succeeds.
\end{lem}
\begin{proof}
	Since only $\mathcal{P}$-requirements reinitialized lower priority requirements, and each can act at most once, every requirement is reinitialized only finitely often.
	
	We first see that every $\mathcal{I}_m$-strategy succeeds. Take a stage after which the strategy is not reinitialized and consider the tuple restrained by the strategy. By the previous Lemma, each of its restrained elements are $E$-distinct, so the strategy succeeds.
	
	Next, consider a $\mathcal{P}_{n,o}$-strategy. Let $s$ be a stage large enough that the strategy is not reinitialized after stage $s$. Let $\bar{a}$ be the full tuple of elements restrained by higher-priority $\mathcal{I}$-strategies (which has settled by stage $s$). Let $\mathcal{Q}_{i_q,j_q,\bar{x}_q}$ for $q< K$ be the collection of higher-priority $\mathcal{Q}$-strategies.
	Suppose that $W_n/E$ is infinite, and let $t>s$ be a stage after which $W_n$ contains at least $\vert \bar{a}\cup\bigcup_{q<K}\bar{x}_q \vert + K+1$ $E$-distinct elements. At any such stage, at most $K$ $E_t$-classes are restrained as parameters $y$ by higher priority $\mathcal{Q}$-strategies, so there must be an unrestrained member of $W_{n,t}$ which the strategy will collapse with $o$ and thus be permanently satisfied.
	
	Finally, we consider a $\mathcal{Q}_{i,j,\bar{x}}$-strategy. We consider the three cases: If $\bar{x}$ is not $E$-distinct, then this is seen at some point and we set $U=\omega$. If $\bar{x}$ is $E$-distinct and $W_j=W_k$, then infinitely often, we add some $n$ to $U$, but then we collapse this $n$ in with $x_0$. So, $U=[\bar{x}]_E$. If $W_j\neq W_k$, then let $s$ be the least stage so that the strategy is not reinitialized after stage $s$ and the approximation says $W_j\neq W_k$ for all $t>s$. Let $y$ be the parameter chosen at stage $s$. Then we need only see that $y\notin [\bar{x}]_E$. We consider what strategy might cause this collapse. It cannot be a higher priority $\mathcal{P}$-requirement, since the strategy is not reinitialized after stage $s$. It cannot be lower priority $\mathcal{P}_{n,o}$-requirements since both $\bar{x}$ and $y$ are restrained by $\mathcal{Q}_{i,j,\bar{x}}$, so neither can be $E_t$-equivalent to the chosen element $z\in W_n$. It cannot be a $\mathcal{Q}$-requirement, since the lower-priority strategy's parameter $y$ cannot be equivalent to either the higher-priority strategy's $y$ or $x$, by the previous Lemma.
\end{proof}

This concludes the proof of Theorem~\ref{dark minimal ceer high2}.	
\end{proof}

\section{Dark jumps}\label{section: dark jumps}

In the remaining three sections, we move out from the realm of ceers and consider equivalence relations of higher complexity. In particular, we now ask how complex an infinite equivalence relation $E$ must be for its jump to be dark.
Clemens, Coskey, and Krakoff \cite[Theorem 4.2]{CCK} show that $E^\dpl$ is light for every infinite ceer $E$ 
and there are infinite $\Delta^0_4$ equivalence relations $E$ so that $E^\dpl$ is dark.
Here we prove that $\Sigma^0_2$ is the lowest arithmetical complexity of an equivalence relation $E$ such that $E^\dpl$ is dark (thus answering~\cite[Question~6]{CCK}).

\smallskip

First, we show that the jump of every infinite $\Pi^0_2$ equivalence relation is light.

\begin{thm}
	If $E\in \Pi^0_2$ is infinite, then $\Id \leq E^\dpl$.
\end{thm}
\begin{proof}
	We let $E_s$ be computable approximations to $E$ so that $x E y$ if and only if there are infinitely many stages $s$ so that $x \mathrel{E_s} y$.
	We construct a uniform sequence of c.e. sets $W_{i_j}$, for $j\in \omega$, so that $W_{i_j}\subsetneq_E W_{i_{j+1}}$, for each $j\in \omega$.
	
	We let $W_{i_0}=\{0\}$. We define $W_{i_{j+1}}$ as follows:
\[	
	x\in W_{i_{j+1}} \mbox{ if and only if } (\forall y< x)(\exists s\geq x)( \exists z\in W_{i_j})(y\,{E_s}\, z).
	\]
	
	\begin{lem}
		If $[0,x)\subseteq_{E} W_{i_j}$, then $x\in W_{i_{j+1}}$.
	\end{lem}
	\begin{proof}
		For each $y< x$, there is a $z\in W_{i_j}$ so that $y\mathrel{E} z$. Thus for infinitely many $s$ we have $y\mathrel{E_s} z$, witnessing $x\in W_{i_{j+1}}$. 
	\end{proof}
	
	\begin{lem}
		Each $W_{i_j}$ is a finite initial segment of $\omega$.
	\end{lem}
	\begin{proof}
		We prove this by induction. This is true for $j=0$.
		
		Fix an element $y\notin [W_{i_j}]_E$. This exists because $W_{i_j}$ is finite and $E$ has infinitely many classes. Then let $s$ be a stage large enough that $W_{i_j}=W_{i_j,s}$ and every $z\in W_{i_j}$ and $t>s$ we have $y\mathrel{\cancel{E_t}} z$. Then no $x>s$ can ever enter $W_{i_{j+1}}$. 
	\end{proof}

	It follows that $W_{i_j}\subsetneq_E W_{i_{j+1}}$ for each $j$, so $j\mapsto i_j$ is a reduction of $\Id$ to $E^\dpl$.
\end{proof}

On the other hand, there are $\Sigma^0_2$ sets whose jumps are dark.

\begin{thm}\label{sigma2 eqrel with dark jump}
	There exists an infinite $\Sigma^0_2$ equivalence relation $E$ so that $\Id\not\leq E^\dpl$.
\end{thm}
\begin{proof}
	We construct $E$ as a c.e. set via a finite injury argument over $\mathbf{0}'$. We have requirements:
	
	\begin{itemize}
		\item[$\mathcal{R}_i$]: If $W_i$ is infinite, then it contains two entries which are $E^{\dpl}$-equivalent.
		\item[$\mathcal{Q}_j$]: There are $x_1,\ldots x_j$ which are $E$-inequivalent.
	\end{itemize}

The $\mathcal{R}$-requirements ensure that there is no reduction from $\Id$ to $E^{\dpl}$, while the $\mathcal{Q}$-requirements obviously ensure that $E$ is infinite.
We place these requirements in order-type $\omega$. A $\mathcal{Q}$-requirement acts by placing a restraint. At every stage $s$, we allow the first $s$ requirements to act in turn. In fact, $\mathcal{R}$-requirements may act at infinitely many stages and cause infinitely many $E$-collapses.

The strategy for an $\mathcal{R}_n$-requirement is as follows:
Let $\bar{x}$ be the tuple of elements restrained by higher-priority $\mathcal{Q}$-requirements. Using $\mathbf{0}'$, we seek a set $\mathcal{I}$ of $3\cdot 2^{\abs{\bar{x}}}+1$ numbers in $W_n$. If there are not this many, then $W_n$ is not infinite and the requirement is satisfied. From these numbers, we use $\mathbf{0}'$ to find four that agree on the (current) classes of $\bar{x}$. That is, for each of these $3\cdot 2^{\abs{\bar{x}}}+1$ indices $j\in W_n$ and $x\in \bar{x}$, we use $\mathbf{0}'$ to ask if any member (there will be only finitely many) of $[x]_{E_s}$ is in $W_j$. Then by the pigeon-hole principle, there are four that give the same answer for every $x\in \bar{x}$. Fix these indices: $j,k,l,m$. If there are two indices $i,i'\in \{j,k,l,m\}$ so that $W_i$ and $W_{i'}$ are contained in $[\bar{x}]_{E_s}$, then $\mathbf{0}'$ sees this and the requirement will be automatically satisfied, so no further action is taken. So, we may suppose $W_j,W_k,W_l$ are each not contained in $[\bar{x}]_{E_s}$.
Note that the family $\{W_j,W_k,W_l\}$ must contain two finite sets or two infinite sets.
We begin with working with the pair $j,k$ and, until proven otherwise, we guess that $W_j$ and $W_k$ are both infinite.

Then, we perform the following \texttt{Collapse}($j,k$) module: 

\begin{itemize}
\item[] At each stage $s$ greater than every $x\in \bar{x}$, we ask $\mathbf{0}'$ if there is a $y\geq s$ which is in $W_j$ and we ask if there is a $y\geq s$ which is in $W_k$. We distinguish two cases.
\begin{enumerate}
\item\label{foundfinitecase} If the answer is no to either, then we stop this module and we call the \texttt{FoundFiniteSet} module instead. 
\item Assuming case (\ref{foundfinitecase}) didn't happen, we now act to ensure that every $z<s$ is either  in both or neither of $[W_j]_{E_s}$ and $[W_k]_{E_s}$. We act successively for each $z\in (\max(\bar{x}),s)$. If $z$ is not least in its $E_s$-equivalence class, then we have already ensured this when previously considering a number $y<s$ which is $E_s$-equivalent to $z$, so we do nothing. Otherwise, we ask $\mathbf{0}'$ if $z\in [W_j]_{E_s}$ and if $z\in [W_k]_{E_s}$. 
\begin{enumerate}
\item If it is in neither or both, we do no action. 
\item If it is in one and not the other, then we find the least $n>s$ in the other set and we $E$-collapse the interval $[z,n]$. 
\end{enumerate}
\end{enumerate}
\end{itemize}

We now describe the \texttt{FoundFiniteSet} module:
\begin{enumerate}
\item If this is the first time we call this procedure, say having found that $W_j$ is finite, then we simply return to the \texttt{Collapse}($k,l$) module (we just assume $W_k$ and $W_l$ are infinite until we see otherwise).
\item If this is the second time we call this procedure, say having found that $W_j$ and $W_k$ are finite, then we simply collapse $[\max(\bar{x})+1,\max(W_j,W_k)]$ to a single $E$-class.
\end{enumerate}



Note that since every collapse involves an interval, the classes of $E$ are intervals as well.

A $\mathcal{Q}_j$ strategy acts as follows: Let $\bar{x}$ be the tuple restrained by $\mathcal{Q}_{j-1}$ (or $\bar{x}=\emptyset$ if $j=0$). Wait to find a stage $s$ and a number $y<s$ so that $y$ is the greatest element of $[\max(\bar{x}+1))]_{E_s}$ and $[y]_{E_s}=[y]_{E_{s-1}}$.
 Once such a $y$ is found, the strategy places a restraint on the tuple $\bar{x}y$.


The strategies are interwoven in priority order: $\mathcal{R}_0<\mathcal{Q}_0<\mathcal{R}_1<\mathcal{Q}_1 < \cdots$. Whenever an $\mathcal{R}$-strategy runs a \texttt{FoundFiniteSet} module, all lower priority strategies are reinitialized. This is the only source of injury. At each stage $s$, we allow the requirements to act in order until one of them ends the stage. A $\mathcal{Q}_j$-strategy which is still waiting to find a $y$ or which acts by declaring its restraint $\bar{x}y$ ends the stage, and a $\mathcal{R}_n$-strategy which runs a \texttt{FoundFiniteSet} module ends the stage.

\begin{lem}
	Suppose that a $\mathcal{Q}_k$ strategy restrains a tuple $\bar{x}y$ at stage $s$, and $t>s$. Then either the strategy has been reinitialized between stages $s$ and $t$ or $\bar{x}y$ are the largest members of the first $\vert \bar{x} y\vert$ $E_t$-equivalence classes.	
	In particular, $[z]_{E_s}=[z]_{E_t}$ for every $z\in \bar{x}y$.
%
\end{lem}
\begin{proof}
	The result holds by induction for every $x_i\in \bar{x}$. Namely, $\bar{x}$ is restrained by the strategy $\mathcal{Q}_{j-1}$ at a stage $r<s$. By inductive hypothesis applying the claim to the $\mathcal{Q}_{j-1}$-strategy, $x_i$ is the greatest number in the $i+1$th $E_t$-class as needed. We must consider the $E$-class of $y$. Since $y \mathrel{E_s} \max(\bar{x})+1$, we need only show that as long as the $\mathcal{Q}_k$ has not been reinitialized, no number $>y$ ever becomes equivalent to $y$. 
	
	Since $[y]_{E_{s-1}}=[y]_{E_s}$, each higher priority $\mathcal{R}$-strategy (without loss of generality, suppose it is running the \texttt{Collapse}($j,k$) module) has considered the class $[y]_{E_s}$ on its previous pass and found that it intersected either both or neither of $W_j$ and $W_k$. Thus, at any future stage $t>s$ where $[y]_{E_t}=[y]_{E_s}$, as long as the strategy remains in the \texttt{Collapse}($j,k$) module, this strategy will never have a need to collapse any element with $y$. If the strategy takes the \texttt{FoundFiniteSet} module, then the $\mathcal{Q}_k$-strategy is reinitialized and the desired result holds. Thus, no higher priority strategy can ever cause the $E$-class of $y$ to grow.
			
	Consider the collapses caused by lower-priority $\mathcal{R}$-strategies at a stage $t>s$ and suppose that we have $[y]_{E_t}=[y]_{E_s}$. 
	The strategy collapses 
	finite intervals of numbers $[z,n]$ which are greater than the largest element in the restrained tuple. Since $y$ is the largest number in its $E_t$-equivalence class, 
	no number in this finite interval can be equivalent to $y$, 
	so this collapse does not add any elements to $y$'s $E$-class.
\end{proof}

\begin{lem}
	Each strategy is satisfied.
\end{lem} 

\begin{proof}
	Each strategy may injure lower priority requirements at most twice (each time it runs the \texttt{FoundFiniteSet} module), so every strategy is reinitialized only finitely often. 
	
	Suppose towards a contradiction that the first strategy that fails is a $\mathcal{R}_n$-strategy. Fix $\bar{x}$ to be the numbers restrained by higher-priority $\mathcal{Q}$-strategies. Then $\mathcal{R}_n$ begins by choosing indices $j,k,l$. Note that for any $x\in \bar{x}$, we have $[x]_{E_s}\cap W_j=\emptyset \leftrightarrow [x]_{E_s}\cap W_k=\emptyset \leftrightarrow [x]_{E_s}\cap W_l=\emptyset$ where $s$ is the stage when $j,k,l$ were chosen after the last time the $\mathcal{R}_n$-strategy is reinitialized. But by the previous claim, 	
	$[x]_{E_s}=[x]_E$, so 
\[	
	[x]_{E}\cap W_j=\emptyset \Leftrightarrow [x]_{E}\cap W_k=\emptyset\Leftrightarrow [x]_{E}\cap W_l=\emptyset.
	\]
	 So, on these classes, the three sets agree.
	
	First suppose that both of $W_j$ and $W_k$ are infinite. We now check that the \texttt{Collapse}$(j,k)$ module ensures that $W_j=_E W_k$. Fix $z>\max(\bar{x})$ (i.e., a class distinct from the ones considered above) and suppose that $z\in [W_j]_E$. Then at some stage $s>z$ we have $z\in [W_j]_{E_s}$. Then at this stage, we ensure that $z\in [W_k]_{E_s}$.  This covers every class by the previous claim, so $j\mathrel{E^\dpl}k$ satisfying the $\mathcal{R}_n$ requirement.
	
	Similarly, if exactly one of $W_j$ or $W_k$ is finite (without loss of generality, assume it is $W_j$), 
	and $W_l$ is infinite then the \texttt{Collapse}$(k,l)$ module ensures that $W_k=_E W_l$. If two of the sets, say $W_j$ and $W_k$ are finite, then the \texttt{FoundFiniteSet} module ensures that $W_j=_E W_k$ since they must both intersect the class of $\max(\bar{x})+1$ (since they were chosen to not be contained in $[\bar{x}]_{E_s}=[\bar{x}]_{E}$) and no larger class. Thus, the strategy succeeds after all.

	Next, suppose towards a contradiction that $\mathcal{Q}_j$ is the first strategy that fails. From the above lemma, we need only show that the wait to find a $y$ as needed must end. At each stage $t$, let $y_t=\max([\max(\bar{x})+1]_{E_t})$. This would work for our choice of $y$ unless $[\max(\bar{x})+1]_{E_t}\neq [\max(\bar{x})+1]_{E_{t-1}}$. This can only happen due to the action of a higher priority $\mathcal{R}$-requirement, since $\mathcal{Q}_j$ ends the stage since it is waiting to find its $y$. We can suppose, without loss of generality, that the higher priority strategy is in a \texttt{Collapse}$(j,k)$ module, since the \texttt{Collapse}$(k,l)$ module is symmetric and it can run the \texttt{FoundFiniteSet} module at most twice. Then growing the $E$-class of $\max(\bar{x})+1$ must be because $\max(\bar{x})+1$ was seen to be in exactly one of $[W_j]_{E_{t-1}}$ or $[W_k]_{E_{t-1}}$. But this can happen only once in the \texttt{Collapse}$(j,k)$ module, since after stage $t$ it is in both. Thus, after finitely many stages, we must have $[\max(\bar{x})+1]_{E_t}=[\max(\bar{x})+1]_{E_{t-1}}$ and $\mathcal{Q}_j$ can choose its element $y$.
\end{proof}
This concludes the proof of Theorem~\ref{sigma2 eqrel with dark jump}.
\end{proof}

\section{Jumps depend on notations}\label{section: jumps depend on notations}

We now consider the transfinite jump hierarchy. Clemens, Coskey, and Krakoff \cite[Question 3]{CCK} ask whether the degree of $E^{\dpl a}$ depends on the notation $a\in \O$ or only the ordinal $\abs{a}$. We show that it does indeed depend on the notation, but we give a bound on how much it can depend on the notation.

\begin{notation}
	To avoid having towers of exponentials to represent successor ordinals, we introduce the function $P(x)=2^x$ and we write $P^{(k)}(x)$ for the $k$th iterate of the function $P$ on $x$. Note that if $n$ is a notation for the ordinal $\alpha$, then $P^{(k)}(n)$ is a notation for the ordinal $\alpha+k$.
\end{notation}

The following observation follows directly from the definitions.

\begin{obs}\label{obs:Uniform reduction to your own jump}
	For any notations $a<_{\mathcal{O}} b$, there is a computable function $f_{a,b}$ so that $f_{a,b}$ witnesses $E^{\dpl a}\leq E^{\dpl b}$ for any equivalence relation $E$. Further, $f_{a,b}$ can be uniformly found from the notations $a$ and $b$.
\end{obs}

The following lemma will be used to manage possible reductions into $E^{\dpl a}$ where $\vert a \vert$ is a limit ordinal.

\begin{lem}\label{lem:computably inseparable classes}
	For any equivalence relation $E$, the classes of $E^\dpl$ are computably inseparable.
\end{lem}
\begin{proof}
	Suppose towards a contradiction that $[i]_{E^\dpl}$ and $[j]_{E^\dpl}$ are separated by the computable set $A$. That is, $[i]_{E^\dpl}\subseteq A$ and $[j]_{E^\dpl}\cap A = \emptyset$.
		By the recursion theorem, we can take an index $e$ so that $W_e=W_i$ if $e\notin A$ and $W_e = W_j$ if $e\in A$. In either case, this gives a contradiction.
\end{proof}

We first consider ordinals $<\omega^2$, and show that the notation does not matter in this case.

\begin{lem}\label{All good below omega squared}
	Let $\alpha$ be an ordinal $<\omega^2$ and $a,b\in \mathcal{O}$ have $\abs{a}=\abs{b}=\alpha$. Then for any $E$, we have $E^{\dpl a}\equiv E^{\dpl b}$.
\end{lem}
\begin{proof}
	

	The proof is by induction on $\alpha$. 
	We note that if the result is shown for $\alpha$, then for any notation $b$ with $\abs{b}>\alpha$ and any notation $a$ with $\abs{a}=\alpha$, then $E^{\dpl a}\leq E^{\dpl b}$. To see this, take the notation $c$ with $c<_{\mathcal O} b$ and $\abs{c}=\alpha$. Then $E^{\dpl a}\equiv E^{\dpl c}\leq E^{\dpl b}$. We call this the ``reduction form'' of the inductive hypothesis.	
	
%

	The lemma clearly holds for all finite $\alpha$. The set of $\alpha$ for which this is true is also clearly closed under successor. It suffices to show the result for limit ordinals $\alpha<\omega^2$.
	
	Let $a=3\cdot 5^i$ and $b=3\cdot 5^j$ be notations for $\omega\cdot n$. Let $c$ be least so that $\abs{\phi_i(c)}\geq \omega\cdot (n-1)$ and $d$ be least be so that $\abs{\phi_j(d))}\geq \abs{\phi_i(c)}$. For every $k>c$, $\phi_i(k)=P^{(z)}(\phi_i(c))$ for some $z$. Similarly, for every $k>d$, $\phi_j(k)=P^{(z)}(\phi_j(d))$ for some $z$.
	
	We build a reduction of $E^{\dpl a}$ to $E^{\dpl b}$ as follows: We send the first $c$ columns of $E^{\dpl a}$ to the columns $d$ through $d+c-1$ of $E^{\dpl b}$. This can be done by the reduction form of the inductive hypothesis since the first $c$ columns of $E^{+a}$ are all $E^{+g}$ for some $g$ with $\abs{g}<\omega\cdot (n-1)$ and the images are of the form $E^{\dpl h}$ where $\abs{h}\geq \omega\cdot (n-1)$.
	
	Next, we send the $c$th column of $E^{\dpl a}$ to the $(d+c)$th column of $E^{\dpl b}$ which again we can do by the reduction form of the inductive hypothesis. To figure out how to send the $c+1$th column, we find the number $k$ so that $\phi_i(c+1)=P^{(k)}(\phi_i(c))$. Then we find the first unused column $e$ in $E^{(b)}$ so that $\phi_j(e)=P^{(l)}(d)$ with $l>k$. We can then use the reduction from $E^{\dpl c}$ to $E^{\dpl d}$ to uniformly find a reduction from $E^{P^{(k)}(c)}$ to $E^{P^{(l)}(d)}$. Repeating as such, we uniformly send every column of $E^{+a}$ into $E^{+b}$ giving the needed reduction. 
\end{proof}

Next we see that notation does matter at $\omega^2$.

\begin{thm}
	For any notation $b$ for $\omega^2$ there exists another notation $a$ for $\omega^2$ so that $\Id^{\dpl a}\not\leq \Id^{\dpl b}$.
	
	There are two notations $a,b$ for $\omega^2$ so that $\Id^{\dpl a}$ and $\Id^{\dpl b}$ are incomparable.
\end{thm}
\begin{proof}
	
	Let $b=3\cdot 5^j$ be a given notation for $\omega^2$.
	
	We take $a=3\cdot 5^e$ for an index $e$ which we control by the recursion theorem. For each $x$, we let $\phi_e(x) = P({3\cdot 5^{i_x}})$ for an infinite sequence of indices $i_x$ which we control by the recursion theorem. Until we determine otherwise, we define, stage by stage that $\phi_{i_{x+1}}(0) = P({3\cdot 5^{i_x}})$ and $\phi_{i_x}(s+1) = P({\phi_{i_x}(s)})$.

	We perform the following actions for the sake of diagonalization.
	To ensure that $\phi_k$ is not a reduction of $\Id^{\dpl a}$ to $\Id^{\dpl b}$, we wait for $\phi_k(\langle k,0\rangle )$ to converge, say to $\langle m,n\rangle$. Since $\vert P({3\cdot 5^{i_x}})\vert $ is a successor ordinal, Lemma~\ref{lem:computably inseparable classes} shows that the classes of $\Id^{\dpl P({3\cdot 5^{i_x}})}$ are computably inseparable. Thus we know that if $\phi_k$ is a reduction, then it must send the entire $k$th column into the $m$th column of $\Id^{\dpl b}$. But the $m$th column of $\Id^{\dpl b}$ is equivalent to $\Id^{+\phi_j(m)}$. So, at the stage $s$ when we see that $\phi_k(\langle k,0\rangle )\downarrow = \langle m,n\rangle$, we make $\phi_{i_k}(s+1)=\phi_{i_k}(s)+_{\mathcal{O}} \phi_{j}(m)+_{\O}1$. This ensures that $\abs{P({3\cdot 5^{i_x}})}>\abs{\phi_j(m)}$. 
	
	For each column, we will only perform this operation once (for all $t>s$, we set $\phi_{i_k}(t+1) = P({\phi_{i_k}(t)})$). Thus, if $3\cdot 5^{i_x}$ is a notation for some limit ordinal less than $\omega^2$, then $3\cdot 5^{i_{x+1}}$ is also a notation for a limit ordinal less than $\omega^2$. Thus, this is true for all $x$ by induction and thus $a$ is a notation for $\omega^2$.
	
	Suppose towards a contradiction that $\phi_k$ is a reduction of $\Id^{\dpl a}$ to $\Id^{\dpl b}$. Then on the $k$th column, $\phi_k$ gives a reduction of $\Id^{\dpl P({3\cdot 5^{i_k}})}$ to $\Id^{\dpl\phi_j(k)}$. Let $c$ be so $c<_\mathcal{O} P({3\cdot 5^{i_x}})$ and $\abs{c}=\abs{\phi_j(k)}$. Then $\Id^{\dpl c}\equiv \Id^{\dpl \phi_j(k)}$ by Lemma~\ref{All good below omega squared}. But then $(\Id^{\dpl c})^\dpl \leq \Id^{\dpl P({3\cdot 5^{i_x}})}\leq \Id^{\dpl \phi_j(k)}\equiv \Id^{\dpl c}$. But then $\Id^{\dpl c}$ $m$-bounds every HYP set  \cite[Theorem 3.10]{CCK}, but this is a contradiction since $\Id^{\dpl c}$ is itself HYP.
	
	Running the same strategy in the reverse direction, we can construct $a$ and $b$ so that $\Id^{\dpl a}$ and $\Id^{\dpl b}$ are incomparable.
\end{proof}

We next see that for any computable ordinal $\alpha$, the equivalence relations $\Id^{\dpl a}$ for $a$ with $\vert a \vert = \alpha$ form a reasonably well bounded collection of equivalence relations. We will need the following observation:

\begin{obs}
	There is a computable function $x\mapsto 2\cdot_{\mathcal{O}} x$ which sends a notation $a$ for $\alpha$ to a notation for $2\cdot \alpha$. Further, $x<_\mathcal{O} 2\cdot_{\mathcal{O}} x$ for every $x\in \mathcal{O}$.
\end{obs}
\begin{proof}
	This is done via transfinite recursion and the recursion theorem. We define $2\cdot_{\mathcal{O}} P(a)$ to be $P^{(2)}({2\cdot_{\mathcal{O}} a})$ and we define $2\cdot_{\mathcal{O}} (3\cdot 5^e)$ as $3\cdot 5^i$ where $\phi_i(x)=2\cdot_{\mathcal{O}} \phi_e(x)$.
\end{proof}

\begin{thm}
	For any recursive ordinal $a$, $\Id^{\dpl a}\leq =_{\Sigma^0_{2\cdot_{\mathcal{O}}a}}$ where $=_{\Sigma^0_c}$ is the equivalence relation of equality of $\Sigma^0_{c}$ sets (given by a notation $c\in \mathcal{O}$).
	
	Further, this is uniform in the notation $a$.
\end{thm}
\begin{proof}
	We prove this by induction on the notation $a$. For the base of the induction, let $a=1$, i.e., the notation for the ordinal $0$. Then $\Id^{\dpl a}=\Id$ and $\Sigma^0_{2\cdot_{\mathcal{O}}a}=\Sigma^0_0$. We can send $n$ to an index for the $\Sigma^0_0$ set $\{n\}$.
	
	Next suppose that $a=P(b)$. Then we assume $\Id^{\dpl b}$ reduces to $=_{\Sigma^0_{2\cdot_{\mathcal{O}}b}}$ sets. Then $\Id^{\dpl a}$  reduces to $(=_{\Sigma^0_{2\cdot_{\mathcal{O}}b}})^\dpl$. Thus it suffices to show the following claim:
	
	\begin{claim}
		For any $c\in \mathcal{O}$, $(=_{\Sigma^0_c})^\dpl\leq =_{\Sigma^0_{P^{(2)}(c)}}$. 
	\end{claim} 
	\begin{proof}
		Let $(S_m)_{m\in \omega}$ be a natural indexing of all $\Sigma^0_c$ sets.
		Let $F$ be a function which sends $i$ to a $\Sigma^0_{P^{(2)}(c)}$-index for the set $\{m \colon \exists k \left(S_m=S_k \wedge k\in W_i \right)\}$, and observe that $F$ is a reduction.
	\end{proof}

	Finally, suppose that $a=3\cdot 5^i$. Then by the assumed uniformity for all ordinal notations $<_\mathcal{O}a$, we have uniform reductions of each $\Id^{\dpl \phi_i(k)}$ to $=_{\Sigma^0_{2\cdot_{\mathcal{O}} \phi_i(k)}}$. Since we can uniformly turn $\Sigma^0_{2\cdot_{\mathcal{O}} \phi_i(k)}$-indices for a set into a $\Sigma^0_{2\cdot_{\mathcal{O}} a}$-index for the same set, we see that each $\Id^{\dpl \phi_i(k)}$ reduces to $=_{\Sigma^0_{2\cdot_{\mathcal{O}}a}}$. By coding on distinct columns, i.e., using the fact that $=_{\Sigma^0_{2\cdot_{\mathcal{O}}a}}\times \Id\leq =_{\Sigma^0_{2\cdot_{\mathcal{O}}a}}$, we see that $\Id^{\dpl a}\leq =_{\Sigma^0_{2\cdot_{\mathcal{O}}a}}$. And again this is uniform.
\end{proof}

\begin{cor}\label{cor:Single ER bound all notations of same ordinal}
	For every computable ordinal $\alpha$, there is an equivalence relation $E$ which is $\Pi^0_{2\cdot \alpha+1}$ so that whenever $a\in \mathcal{O}$ is a notation for $\alpha$, we have $\Id^{\dpl a}\leq E$.
\end{cor}
\begin{proof}
	By Spector's uniqueness theorem \cite[Thm 4.5]{sacks1990higher}, if $\vert a \vert = \vert b \vert$, then $H(a)\equiv H(b)$. Further, this is uniform. Thus for any $b$ with $\vert b \vert = \vert a \vert$, we can uniformly turn a $\Sigma^0_{2\cdot_{\mathcal{O}} b}$-index for a set into a $\Sigma^0_{2\cdot_{\mathcal{O}} a}$-index for the same set. Thus fixing any chosen notation $e$ for $\alpha$, for any notation $a$ for $\alpha$, $\Id^{\dpl a} \leq =_{\Sigma^0_{2\cdot_{\mathcal{O}} e}}$ and $=_{\Sigma^0_{2\cdot_{\mathcal{O}} e}}\in \Pi^0_{2\alpha+1}$.
\end{proof}

\section{Every HYP equivalence relation reduces to some $\Id^{\dpl a}$}\label{section: every HYP ER reduces to some Id^+a}

Friedman and Stanley~\cite{FriedmanStanley} proved that the collection of transfinite jumps of the identity relation on reals form a cofinal family in the Borel hierarchy of all Borel isomorphism relations. In this final section, we offer an effective analogue of this result. Namely, we will prove that any HYP equivalence relation is bounded by some $\Id^{\dpl a}$. 

As for many other places of this paper, our starting point is \cite{CCK}. We give a definition of a strong way to reduce a set $A\subseteq \omega$ to an equivalence relation $E$. This is similar to and inspired by \cite[Definition 3.3]{CCK}; whereas they aren't concerned with the image $h(x)$ if $x\notin A$ (so long as it is $E$-contained in the image of the reduction for an $x\in A$), we demand only two possible images depending on whether or not $x\in A$.

Observe that the cross product $E\times \Id$ (as defined in the preliminaries) is equivalent to a uniform join of $E$ with itself countably many times.

\begin{defn}
	A set $A$ \emph{\SSR} to $E^\dpl$ if there is a computable function $h$ and a pair $i,j$ so that $W_i \subseteq_E W_j$,  $h(x)\mathrel{E^\dpl }j$ for every $x\in A$, and $h(x) \mathrel{E^\dpl} i$ for every $x\notin A$. 
\end{defn}

This form of reduction is strong enough to give us a way to transfer set reductions to $\Id^{\dpl a}$ into equivalence relation reductions to $\Id^{\dpl a}$. In the following lemma and throughout this section, we focus on equivalence relations $E$ so that $E\times \Id\leq E$. This is a reasonable assumption since we are trying to build reductions into equivalence relations of the form $\Id^{\dpl a}$ and all such equivalence relations satisfy $E\times \Id\leq E$ \cite[Corollary 2.9]{CCK}.

\begin{lem}\label{From SSR to ER reduction}
	Suppose that $R$ is an equivalence relation and let $A:=\{\langle x,y\rangle : {x\mathrel{R}y}\}$. Suppose that either $A$ or the complement of $A$ \SSR to $E^\dpl$. Suppose further that $E\times \Id \leq E$. Then $R\leq E^\dpl$.
\end{lem}
\begin{proof}
	Let $(h,i,j)$ witness that $A$ or its complement \SSR to $E^\dpl$.
	
	For each $x\in \omega$, let $(h_x,i_x,j_x)$ witness that $A$ or its complement \SSR to the $x$th column of $(E\times \Id)^\dpl$. That is, $h_x(a)=\{\langle x, y\rangle \mid y\in W_{h(a)}\}$,  $W_{i_x}=\{\langle x , y\rangle \mid y\in W_i\}$, and $W_{j_x}=\{\langle x , y\rangle \mid y\in W_j\}$. 
	
	For each $x\in \omega$, let $e_x$ be a c.e. index for the set $\bigcup_{y\in \omega} W_{h_y(\langle x,y\rangle)}$. Each $W_{h_y(\langle x,y\rangle)}$ is contained in the $y$th column and either has the same $E\times \Id$-closure as $W_{i_y}$ or $W_{j_y}$. We now check that $x\mapsto e_x$ is a reduction of $R$ to $(E\times \Id)^\dpl$.
		
	If $a\mathrel{R} b$, then $\{y \mid y\mathrel{R} a \}=\{y \mid y\mathrel{R} b\}$. Similarly, $\{y \mid y\mathrel{\cancel{R}} a\}. = \{y \mid y\mathrel{\cancel{R}} b\}$. So,
	for every $y$, $W_{h_y(\langle a,y\rangle)}$ has the same $E\times\Id$-closure as $W_{h_y(\langle b,y\rangle)}$, so $e_a \mathrel{(E\times\Id)^\dpl} e_b$. If $a\mathrel{\cancel{R}} b$ then $W_{h_a(\langle a,a\rangle)}$ has the same $E\times \Id$-closure as $W_{j_a}$ (or $W_{i_a}$ if it is the complement of $A$ which \SSR to $E^\dpl$), but $W_{h_a(\langle b,a \rangle)}$ has the same $E\times \Id$-closure as $W_{i_a}$ (or $W_{j_a}$ if it is the complement of $A$ which \SSR to $E^\dpl$) showing that $e_a\mathrel{\cancel{(E\times \Id)^\dpl}}e_b$. Thus $x\mapsto e_x$ is a reduction of $R$ to $(E\times \Id)^\dpl$, which is equivalent to $E^\dpl$.
\end{proof}

We note the similarity between the above and the fact that every $\Sigma^0_1$ equivalence relation $E$ reduces to $\Id^\dpl$. That is proved by sending $x$ to $[x]_E$. This is essentially what we do here, but instead of putting $y$ into the set when $y$ is equivalent to $x$, we put $W_{j_y}$ into the set if $y$ is equivalent to $x$.

Below, it will be convenient to reduce into $E\times \Id$ instead of $E$. The following lemma shows how to return to $E$.

\begin{lem}\label{Going from E times Id to E}
	Let $R\leq E$. 
Suppose that $A$ \SSR to $R^\dpl$, then $A$ \SSR to $E^\dpl$.
Similarly, suppose that $A$ \SSR to $R^{\dpl\dpl}$ then $A$ \SSR to $E^{\dpl\dpl}$.
\end{lem}
\begin{proof}
	Let $g$ be a reduction of $R$ to $E$. Take $(h,i,j)$ witnessing that $A$ \SSR to $R^\dpl$. Then we define $f(n) = e_n$ so that $W_{e_n} = \{g(x) \mid x\in W_{h(n)}\}$ Let $W_a=\{g(x) \mid x\in W_{i}\}$ and $W_b=\{g(x) : x\in W_j\}$. Then $(f,a,b)$ \SSR $A$ to $E^\dpl$.
	
	The second case is the same, except we let $f(n)$ be so $W_{f(n)}= \{e_m \mid m\in W_{h(n)}\}$ where $W_{e_m}=\{g(x) \mid x\in W_m\}$.
\end{proof}

In what follows, we will focus on the collection of sets which \SSR to an equivalence relation $\Id^{\dpl a}$, since we now know that, by Lemma~\ref{From SSR to ER reduction}, we can transfer \SSNs to equivalence relation reductions. The following easy fact will serve as the base of our induction.

\begin{lem}\label{baseofinduction}
	Every $\Sigma^0_1$ set \SSR to $\Id^{\dpl}$.
\end{lem}
\begin{proof}
	Fix $S$ a c.e. set.
	Let $i$ be a c.e. index for the empty set and $j$ be a c.e. index for $\omega$. Let $h(x)$ be an index for an enumeration which either gives $\emptyset$ or $\omega$ depending on whether or not we see $x\in S$.
\end{proof}

Next we give an induction which covers every arithmetical equivalence relation.

\begin{lem}\label{Sigma^0_2 induction}
	Suppose that $A$ \SSR to $E^\dpl$. Further suppose that for every $n$ and $p$, the set $\{q \mid A(\langle n,p,q \rangle)\}$ is an initial subset of $\omega$. Finally, suppose that $E\times \Id \leq E$. Then $B(n):= \exists p \forall q A(n,p,q)$ \SSR $E^{\dpl\dpl}$.
\end{lem}
\begin{proof}
	Fix $(h,i,j)$ witnessing $A$ \SSR to $E^{\dpl}$. This shows $A$ \SSR to every column of $E\times \Id$. That is, we have functions $h_x$ and indices $i_x$ and $j_x$ as above so that $W_{i_x},W_{j_x}\subseteq \omega^{[x]}$, $W_{i_x}\subsetneq_{E\times \Id} W_{j_x}$ and $h_x(y)\mathrel{(E\times \Id)^\dpl} i_x$ if $y\notin A$ and $h_x(y) \mathrel{(E\times \Id)^\dpl} j_x$ if $y\in A$.
	
	For each $n$, we let $W_{e_n}$ be a collection containing:
	\begin{enumerate}
		\item For every $y\in \omega$, a c.e. index for the set $\bigcup_{x<y} W_{j_s}\cup \bigcup_{x\geq y} W_{i_x}$
		\item For every $p\in \omega$, a c.e. index for the set $\bigcup_{x\in \omega} W_{h_x(\langle n, p ,x\rangle )}$.
	\end{enumerate}

	Since for every pair $n,p$, the set of $q$ so that $h_x(\langle n,p,x\rangle) \mathrel{E^\dpl} j_x$ is an initial segment of $\omega$, the sets in the second bullet are either already enumerated in the first bullet or are exactly equal to $\bigcup_{x\in \omega} W_{j_x}$.
	
	Finally, take the map $g:n\mapsto e_n$, let $a$ be a c.e. index for just the sets in (1), and let $b$ be a c.e. index for the sets in (1) along with the set $\bigcup_{x\in \omega} W_{j_x}$. Then $(g,a,b)$ \SSR $B$ to $(E\times \Id)^{\dpl\dpl}$. Thus, Lemma~\ref{Going from E times Id to E} shows that $B$ \SSR to $E^{\dpl\dpl}$.
\end{proof}

\begin{thm}\label{completeness at odd levels}
	For every $n\in \omega$, every $\Sigma^0_{2n-1}$ and $\Pi^0_{2n-1}$ equivalence relation 
	reduces to $\Id^{\dpl\hat n}$. 
\end{thm}
\begin{proof}
	
	We first show that for every $n\in \omega$, every $\Sigma^0_{2n-1}$ set \SSR to $\Id^{\dpl \hat n}$. We use Lemma~\ref{baseofinduction} as the base of this induction.
	
	Let $X$ be a $\Sigma^0_{2n+1}$ set. Write $X(n)= \exists p \forall q A(\langle n,p,q\rangle )$. Rewrite this definition as: $X(n) = \exists p \forall q (\forall m<q A(\langle n, p, m\rangle))$. We observe that $\forall m<q A(\langle n,p, m \rangle)$ is a $\Sigma^0_{2n-1}$ set. Thus, it \SSR to $\Id^{\dpl\widehat{n-1}}$ by inductive hypothesis and, by Lemma~\ref{Sigma^0_2 induction}, $X$ \SSR to $\Id^{\dpl \hat n}$. Note that the hypotheses that $\Id^{\dpl \widehat{n-1}}\times \Id\leq \Id^{\dpl \widehat{n-1}}$ holds by \cite[Corollary 2.9]{CCK}.
	
	Finally, applying Lemma~\ref{From SSR to ER reduction} shows that if $R$ is a $\Sigma^0_{2n-1}$ or $\Pi^0_{2n-1}$ equivalence relation, then $R\leq \Id^{\dpl \hat n}$.
\end{proof}

\begin{cor}\label{Jumps of Id aren't simple}
	The equivalence relation $\Id^{\dpl \hat n}$ is not $\Pi^0_{2n-1}$ or $\Sigma^0_{2n-1}$.
\end{cor}
\begin{proof}
	It is easy to see that there are equivalence relations which are $\Sigma^0_{2n-1}$ and not $\Pi^{0}_{2n-1}$ (consider 1-dimensional equivalence relations with a single class comprised of a $\Sigma^0_{2n-1}$-complete set) and similarly equivalence relations which are $\Pi^0_{2n-1}$ and not $\Sigma^0_{2n-1}$. If $\Id^{\dpl \hat n}$ were $\Sigma^0_{2n-1}$, then every $\Pi^0_{2n-1}$-equivalence relations would have to be $\Sigma^0_{2n-1}$ by virtue of reducing to $\Id^{\dpl \hat n}$. Similarly we get a contradiction if $\Id^{\dpl \hat n}$ were $\Pi^0_{2n-1}$.
\end{proof}

We note that Theorem~\ref{completeness at odd levels} is sharp on the scale of the arithmetical hierarchy since $\Id^{\dpl \hat n}$ is a $\Pi^0_{2n}$ equivalence relation and thus there is a $\Delta^0_{2n}$ equivalence relation which does not reduce to $\Id^{\dpl \hat n}$ \cite{IanovskiMillerNgNies}. We can look closer using the Ershov hierarchy:


\begin{thm}
There is a d-c.e.\ equivalence relation $E$ so that $E\not\leq \Id^{\dpl}$.
\end{thm}

\begin{proof}
We partition the odd numbers into countably many sets $S_e$ for $e\in \omega$.
Let $z_{\langle e,i\rangle}$ be the $i$th element of $S_e$. We construct a d-c.e.\ equivalence relation $E$ by stages. We never make any pair of even numbers $E$-equivalent. We may make elements of $S_e$ be $E$-equivalent to $4e$ or $4e+2$ or neither. 

We satisfy the following requirements:
\begin{itemize}
\item[$\mathcal{R}_e$]: $\phi_e$ is not a reduction of $E$ to $\Id^{\dpl}$. 
\end{itemize}
 The strategy for meeting the $\mathcal{R}$-requirements is twofold. On the one hand, we ensure that $4e \mathrel{\cancel{E}} 4e+2$, for all $e$ (in fact every pair of even numbers are $E$-inequivalent). This action forces $W_{\phi_e(4e)}\neq W_{\phi_e(4e+2)}$, otherwise $\phi_e$ would not be a reduction. But, on the other hand, we use the $z_{\langle e,i\rangle}$'s to gradually copy $W_{\phi_e(4e)}$ into $W_{\phi_e(4e+2)}$ and vice versa. Let's discuss in more detail the  module for diagonalizing against a potential reduction $\phi_e$: 
 
 Let $e_0=4e$ and $e_1=4e+2$.

\begin{enumerate}
\item If at some stage $s$ a number $w$ appears in $W_{\phi_e(e_k)}$, for $k\in\{0,1\}$, we take the least unused $z_{\langle e,i\rangle}$ and we let $e_k \mathrel{E} z_{\langle e,i\rangle}$.
\item We wait to see if $w$ appears in  $W_{\phi_e(z_{\langle e,i\rangle})}$. If this happens, we declare $e_k \mathrel{\cancel{E}} z_{\langle e,i\rangle}$ and we let $e_{1-k} \mathrel{E} z_{\langle e,i\rangle}$ instead. 
\end{enumerate}
\smallskip

Now, towards a contradiction, suppose that there is a reduction $\phi_j$ from $E$ to $\Id^{\dpl}$. Since the construction ensures that $4j \mathrel{\cancel{E}} 4j+2$, it must be the case that $W_{\phi_j(j_0)}\neq W_{\phi_j(j_1)}$. Without loss of generality, let $v \in W_{\phi_j(4j)}\smallsetminus W_{\phi_j(4j+2)}$. But then, by item (1) of the module, we have that, at some stage $s$,  $4j$ is  $E$-collapsed with some $z_{\langle j,i\rangle}$.  Observe that, after this collapse, $v$ must enter in $W_{\phi_j(z_{\langle j,i\rangle})}$ (as otherwise, we would have that $W_{\phi_j(z_{\langle j,i\rangle})}\neq W_{\phi_j(4j)}$ but $4j \mathrel{E} z_{\langle j,i\rangle}$, a contradiction). When this happens, by item (2), we make $4j \mathrel{\cancel{E}} z_{\langle j,i\rangle}$ and we let $4j+2 \mathrel{E} z_{\langle j,i\rangle}$ instead. This action guarantees that there is a  stage at which $v$ appears in $W_{\phi_j(4j+2)}$ (as otherwise, $W_{\phi_j(z_{\langle j,i\rangle})}\neq W_{\phi_j(4j+2)}$ but $4j+2 \mathrel{E} z_{\langle j,i\rangle}$), contradicting the assumption that $v \in W_{\phi_j(4j)}\smallsetminus W_{\phi_j(4j+2)}$.

Finally, it immediately follows from the construction that $E$ is d-c.e., since there is no pair of numbers on which $E$ makes more than two mind changes.
%
%
\end{proof}


Theorem~\ref{completeness at odd levels} gives a nice way to represent the arithmetical equivalence relations in terms of FS-jumps, but it is not sharp at the even layers. For example, every $\Sigma^0_2$ and $\Pi^0_2$ equivalence relation reduces to $\Id^{\dpl 2}$, but $\Id^{\dpl 2}$ is $\Pi^0_4$ and we should expect to find a $\Pi^0_3$ equivalence relation that is universal for all $\Sigma^0_2$ and $\Pi^0_2$-equivalence relations. The next lemma gives us an analogous result at the even layers of the arithmetical hierarchy.

\begin{lem}
	Let $Z$ be a universal $\Pi^0_1$-equivalence relation (which exists by \cite[Theorem 3.3]{IanovskiMillerNgNies}). Then every $\Sigma^0_{2n}$ and $\Pi^0_{2n}$ equivalence relation reduces to $Z^{\dpl \hat n}$.
\end{lem}
\begin{proof}
	We first observe that since $Z$ is $\Pi^0_1$-universal and $Z\times \Id$ is $\Pi^0_1$, we have $Z\times \Id\leq Z$. Thus $Z^{\dpl a}\times \Id\leq Z^{\dpl a}$ for any $a\in \O$ \cite[Proposition 2.8]{CCK}.
	
	As above, we will first show by induction that every $\Sigma^0_{2n}$ set \SSR to $Z^{\dpl \hat n}$.
	As the base of our induction, we first show that every $\Sigma^0_2$ set \SSR to $Z^\dpl$. To see this, we fix a $\Sigma^0_2$ set $A$ and we construct a $\Pi^0_1$-equivalence relation $Y$ and show that $A$ \SSR to $Y^\dpl$. This suffices by Lemma~\ref{Going from E times Id to E}.
	
	We fix an computable approximation $(A_s)_{s\in \omega}$ to $A$ so that $x\in A$ if and only if $x\in A_s$ for all sufficiently large $s$. We build a reduction by sending every $x$ to an index $e_x$ which we control by the recursion theorem. We enumerate the complement of $[0]_Y$ into each $W_{e_x}$.
	At stages $s$ when $x\in A_s$, we take a fresh number $z$ and enumerate $z$ into $W_{e_x}$. If at a later stage $t>s$ we have $x\notin A_t$ then we make $z\mathrel{\cancel{Y}} 0$. In fact, we make $[z]_Y=\{z\}$. If we never see such a $t$, we will maintain $z \mathrel{Y} 0$. If $x\notin A$, then $[W_{e_x}]_Y=\omega\smallsetminus [0]_Y$. If $x\in A$, then $[W_{e_x}]_Y=\omega$. This shows that every $\Sigma^0_2$ set \SSR to $Z^\dpl$.
	
	As the step of our induction, we apply Lemma~\ref{Sigma^0_2 induction} as in the proof of Theorem~\ref{completeness at odd levels} using the fact that $Z^{\dpl\hat n}\times \Id\leq Z^{\dpl\hat  n}$.
	
	Finally, Lemma~\ref{From SSR to ER reduction} shows that every $\Sigma^0_{2n}$ or $\Pi^0_{2n}$ equivalence relation reduces to $Z^{\dpl \hat n}$.
\end{proof}

To move to transfinite levels in the HYP hierarchy, we show that we can handle negations and effective unions.

\begin{lem}\label{negations}
	If $A$ \SSR to $E^\dpl$ then the complement of $A$ \SSR to $E^{\dpl\dpl}$.
\end{lem}
\begin{proof}
	Let $(h,i,j)$ witness that $A$ \SSR to $E^\dpl$. Then let $W_{g(x)}$ enumerate the collection of all c.e. supersets of $W_{h(x)}$. Let $W_a$ be an index for the collection of all c.e. supersets of $W_i$ and $W_b$ be an index for the collection of all c.e. supersets of $W_j$. Then $(g,b,a)$ witnesses that the complement of $A$ \SSR to $E^{\dpl\dpl}$.
\end{proof}

\begin{lem}\label{unions}
	Suppose that each member of $(A_k)_{k\in \omega}$ uniformly \SSR to $E^\dpl$ via $(h_k,i_k,j_k)$. Further suppose that $E\times \Id \leq E$. Let $B(n)$ hold if and only if $\exists k A_k(n)$. Then $B$ \SSR $E^{\dpl\dpl}$.
\end{lem}
\begin{proof}
	As above, for each $x$, let $g_x$ be the function showing that $A_x$ \SSR to $(E\times \Id)^\dpl$ using only the $x$th column. That is, 
$W_{g_x(n)}=\{\langle x,y\rangle \mid y\in W_{g_x(n)}\}$.
	
	We first show that $B$ \SSR to $(E\times \Id)^{\dpl\dpl}$.
	
	For each $x\in \omega$, let $W^x_i$ be the set $\{\langle x, y \rangle \mid y\in W_{i_x}\}$. Similarly for $W_j^x$. Finally, let $V^x=W_i^x\cup \bigcup_{y\neq x}W_j^y$.
	
	Let $f(n)$ be a c.e. index for a set which contains indices for every $V^x$ and also contains indices for the sets $V^x\cup W_{g_x(n)}$. 
	If $n\notin B$, then for every $n$, $g_x(n)$ is an index for $W_i^x$, so each set $V^x\cup W_{g_x(n)}$ is 
	 a copy of $V^x$. So, the family is exactly the collection of $V^x$'s. If $n\in B$, then for some $n$ we have $W_{g_x(n)}=W_j^x$, so  $V^x\cup W_{g_x(n)}=\bigcup_{z\in \omega} W_j^z$.

	Finally, $E\times \Id \leq E$ gives the result by Lemma~\ref{Going from E times Id to E}.
\end{proof}

At this point, we can take effective unions and we can take negations. That's all we need to induct up the HYP hierarchy:

\begin{lem}\label{lem:Id jumps are cofinal in HYP ERs}
	Every HYP set \SSR to $\Id^{\dpl a}$ for some $a\in \mathcal{O}$.
\end{lem}
\begin{proof}
	We proceed by induction on notations for computable ordinals with the base case done by Lemma~\ref{baseofinduction}.
	
	Formally, we show that for every notation $c$ for an ordinal $\alpha$, there is some $a$ so that every $\Sigma^0_\alpha$ set uniformly \SSR to $\Id^{\dpl a}$ (i.e. we can find the index of the witness $(h,i,j)$ uniformly from an index of $A$ as a $\Sigma^0_c$ set). Further our construction will produce a computable function $H$ going from $c$ to the notation $a$. Further, whenever $c<_\mathcal O d$, we will have $H(c)<_\mathcal O H(d)$.

\subsubsection*{Successor step} Suppose every $\Sigma^0_\alpha$ set uniformly \SSR to $\Id^{\dpl a}$. Then every $\Pi^0_\alpha$ set uniformly \SSR to $\Id^{\dpl P(a)}$ by Lemma~\ref{negations}. Let $A$ be a $\Sigma^0_{\alpha+1}$ set. Then $A$ is an effective union of $\Pi^0_{\alpha}$ sets. Thus $A$ \SSR to $\Id^{\dpl P^{(2)}(a)}$ by Lemma~\ref{unions}, and this argument is uniform.

\subsubsection*{Limit step} Let $c=3\cdot 5^i$. Then we let $a=3\cdot 5^e$ where $\phi_e(n)=H(\phi_i(n))$. Since by the inductive hypothesis, we know that $H(\phi_i(n))<_\mathcal O H(\phi_i(n+1))$ for every $n\in \omega$, we have $a\in \mathcal O$.

%

If $A$ is a $\Sigma^0_c$ set, then it is an effective union of $\Sigma^0_b$ sets for $b<_{\mathcal{O}} c$. Each of these uniformly \SSR to $\Id^{\dpl P(a)}$ by the uniformity in Observation \ref{obs:Uniform reduction to your own jump} and Lemma \ref{Going from E times Id to E}. 
So, the effective union \SSR to $\Id^{\dpl P^{(2)}(a)}$ by Lemma~\ref{unions}. This argument is uniform, and we can let $H(c) = P^{(2)}(a)$.
\end{proof}

\begin{cor}
	Every HYP equivalence relation reduces to $\Id^{\dpl a}$ for some $a\in \mathcal{O}$.
\end{cor}
\begin{proof}
	Combine the above with Lemma~\ref{From SSR to ER reduction}.
\end{proof}

\begin{cor}
	Every HYP equivalence relation reduces to $=_{\Sigma^0_{a}}$ for some $a\in \mathcal{O}$. The degree of this only depends on the ordinal $\vert a \vert$.
\end{cor}
\begin{proof}
	Combine the above with Lemma~\ref{cor:Single ER bound all notations of same ordinal}.
\end{proof}

\begin{thm}
	If $E^{\dpl}\leq E$, then $E$ is $\geq$ every HYP equivalence relation.
\end{thm}
\begin{proof}
	If $E^\dpl\leq E$, then $E$ is above $\Id^{\dpl a}$ for every $a\in \mathcal O$ by \cite[Propositions 2.3 and 2.7]{CCK}. So this follows immediately from Lemma~\ref{lem:Id jumps are cofinal in HYP ERs}.
\end{proof}

\bibliographystyle{alpha}
\bibliography{refs}

\end{document}